\chardef\other=12
\def\undocatcodespecials{\catcode`\\=\other     \catcode`\{=\other
  \catcode`\}=\other     \catcode`\<=\other     \catcode`\$=\other
  \catcode`\%=\other     \catcode`\~=\other     \catcode`\^=\other
  \catcode`\_=\other     \catcode`\*=\other     \catcode`\`=\other
  \catcode`\!=\other     \catcode`\"=\other     \catcode`\&=\other
  \catcode`\#=\other     \catcode`\|=\other}
\def\ttindent{5mm} % indentation amount of verbatim examples
{\obeyspaces\global\let =\ }
{\obeylines\gdef\obeylines{\catcode`^^M=\active}\gdef^^M{\par}%
\catcode`#=\active \catcode`&=6 \gdef#{\char35}%
 \gdef\ttverbatim{\begingroup\undocatcodespecials \catcode`\#=\active%
   \parindent 0pt \def\_^^M{\allowbreak}\def\${$}\def\`{`}%
   \def\par{\ifvmode\allowbreak\vskip 1pc plus 1pt\else\endgraf\penalty100\relax\fi}%
   \obeyspaces \obeylines \tt}}
\outer\def\begintt{\par
  \begingroup\advance\leftskip by \ttindent
  \ttverbatim \parskip=0pt \catcode`\|=0 \rightskip-5pc \ttfinish}
{\catcode`\|=0 |catcode`|\=\other % | is temporary escape character
  |obeylines % end of line is active
  |gdef||{|char124} %
  |gdef|ttfinish#1^^M#2\endtt{#1|medskip{#2}|endgroup %
  |endgroup%
  |vskip-|parskip|medskip|noindent|ignorespaces}}

\outer\def\beginexample{\par
  \begingroup\advance\leftskip by \ttindent
  \ttverbatim \parskip=0pt \catcode`\|=0 \rightskip-5pc \examplefinish}
{\catcode`\|=0 |catcode`|\=\other % | is temporary escape character
  |obeylines % end of line is active
  |gdef||{|char124} %
  |gdef|examplefinish#1^^M#2\endexample{#1|medskip{#2}|endgroup %
  |endgroup%
  |vskip-|parskip|medskip|noindent|ignorespaces}}

\documentclass[11pt]{article}

\usepackage{amsfonts}
\usepackage{amssymb}
\usepackage{graphics}
\usepackage{amsmath,epsfig,psfrag}
\usepackage[english]{babel}
\usepackage[all]{xy}
\usepackage{amscd}
\usepackage{amsthm}
\usepackage{latexsym}
\usepackage{graphicx}
\usepackage{mathrsfs}

\def\Z{{\Bbb Z}}

\newtheorem{theorem}{Theorem}[section]
\newtheorem{corollary}[theorem]{Corollary}
\newtheorem{proposition}[theorem]{Proposition}
\newtheorem{lemma}[theorem]{Lemma}

\theoremstyle{definition}
\newtheorem{example}[theorem]{Example}
\newtheorem{definition}[theorem]{Definition}

\theoremstyle{remark}
\newtheorem{remark}[theorem]{Remark}

\begin{document}
\title{{Embedding surfaces into $S^3$ with maximum symmetry}}
\author{Chao Wang, Shicheng Wang, Yimu Zhang\\ and \\Bruno Zimmermann}

\date{}
\maketitle

\centerline{Department of Mathematics} \centerline{Peking
University, Beijing, 100871 CHINA} \centerline{and} \centerline{
Universita degli Studi di Trieste, Trieste, 34100 Trieste ITALY}

\begin{abstract}
We restrict our discussion to the orientable category. For $g > 1$,
let $OE_g$ be the maximum order of a finite group $G$ acting on the
closed surface $\Sigma_g$ of genus $g$ which extends over $(S^3,
\Sigma_g)$, where the maximum is taken over all possible embeddings $\Sigma_g\hookrightarrow
S^3$. We will determine $OE_g$ for each $g$, indeed the action realizing  $OE_g$.

In particular,
 with 23 exceptions, $OE_g$ is $4(g+1)$  if $g\ne k^2$ or  $4(\sqrt{g}+1)^2$
if $g=k^2$,
and moreover $OE_g$ can be realized by unknotted embeddings for all $g$
except for $g=21$ and $481$.

\end{abstract}

\tableofcontents

\section{Introduction}
Surfaces belong to  the most  familiar topological subjects to us,
mostly because we can see them staying in our 3-space in various
manners. The symmetries of the surfaces have been studied for a long
time, and it will be natural to wonder when these symmetries can be
embedded into the symmetries of our 3-space (3-sphere). In
particular, what are the orders of the maximum symmetries of
surfaces which can be embedded into the symmetries of the 3-sphere
$S^3$? We will solve this maximum order problem  in this paper in
the orientable category.

We use $\Sigma_g$ ($V_g$) to denote the closed orientable surface
(handlebody) of genus $g>1$, and  $G$ to denote a finite group
acting on $\Sigma_g$ or on an orientable 3-manifold. The actions we
consider are always faithful and orientation-preserving on both
surfaces and 3-manifolds. We are always working in  the smooth
category. By the geometrization of finite group actions in dimension
3, for actions on the 3-sphere, we can then restrict to orthogonal
actions.

Let $O_g$  be the maximal order of all finite  groups which can act
on $\Sigma_g$. A classical result of Hurwitz states that $O_g$ is at
most $84(g-1)$ \cite{Hu}. However, to determine $O_g$ is still a
hard and famous question in general, and there are numerous
interesting partial results.

Let $OH_g$  be the maximal order of all finite  groups which can act
on $V_g$. It is a result due to Zimmermann \cite{Zi} that $4(g+1)\le
OH_g\le 12(g-1)$, see also  \cite{MMZ}, moreover $OH_g$ is either
$12(g-1)$ or $8(g-1)$ if $g$ is odd, and each of them is achieved by
infinitely many odd $g$ \cite{MZ}. However, in general $OH_g$ are
still not determined either.

In \cite{WWZZ}, we considered finite group actions on the pair
$(S^3, \Sigma_g)$, with respect to an embedding $e:
\Sigma_g\hookrightarrow S^3$. If $G$ can act on the pair $(S^3,
\Sigma_g)$ such that its restriction on $\Sigma_g$ is the given
$G$-action on $\Sigma_g$, we call the action of $G$ on $\Sigma_g$
extendable (over $S^3$ with respect to $e$).

Call an embedding $e: \Sigma_g\hookrightarrow S^3$ unknotted, if
each component of $S^3\setminus e(\Sigma_g)$ is a handlebody,
otherwise it is knotted. Similarly, we define an action of $G$ on
$V_g$ to  be extendable and the embedding $e:V_g\hookrightarrow S^3$
to be unknotted or knotted. For each $g$, the unknotted embedding is
unique up to isotopy of $S^3$ and automorphisms of $\Sigma_g$ (resp.
$V_g$).

Let $OE_g$  be the maximal order of all finite groups admitting an action on $\Sigma_g$ which
is extendable over $S^3$ w.r.t. $e$ for some $e$. Let $OE^u_g$  be the maximal order of all
finite group actions on $\Sigma_g$ which is extendable over  $S^3$ w.r.t.
some unknotted embedding. Then we know that $4(g+1)\le OE^u_g\le
OH_g\le 12(g-1)$, and there are only finitely many $g$ such that
$OE^u_g = 12(g-1)$; moreover, $OE^u_g\ge 4(n+1)^2$ for each $g=n^2$
\cite{WWZZ}.

In this paper we will determine $OE_g$ for all $g > 1$ (Theorem
\ref{OE}). We can also determine $OE^u_g$ and $OE^k_g$ (Theorem
\ref{OES} and Theorem \ref{OEK}), where $OE^k_g$ denotes the maximal
order of finite group actions on $\Sigma_g$ which extend over $S^3$
w.r.t. all possible knotted embeddings.

\begin{theorem}\label{OE}
The maximal orders $OE_g$ are given in the following table.

\vskip 0.1true cm

\begin{center}
\begin{tabular}{|c|c|}
\hline $OE_g$ & $g$\\\hline $12(g-1)$ & $2, 3, 4, 5, 6, 9, 11, 17
,25, 97, 121, 241, 601$\\\hline $8(g-1)$ & $7, 49, 73$\\\hline
$20(g-1)/3$ & $16, 19, 361$\\\hline $6(g-1)$ & $21, 481$\\\hline
$192$ & $41$\\\hline $7200$ & $1681$\\\hline $4(\sqrt{g}+1)^2$ &
$g=k^2, k\neq3, 5, 7, 11, 19, 41$\\\hline $4(g+1)$ & the remaining
numbers\\\hline
\end{tabular}
\end{center}

\vspace{5pt}
\end{theorem}

\begin{theorem}\label{OES}
The maximal orders $OE^u_g$ are given in the following table.

\vskip 0.1true cm

\begin{center}
\begin{tabular}{|c|c|}
\hline $OE^u_g$ & $g$\\\hline $12(g-1)$ & $2, 3, 4, 5, 6, 9, 11, 17
,25, 97, 121, 241, 601$\\\hline $8(g-1)$ & $7, 49, 73$\\\hline
$20(g-1)/3$ & $16, 19, 361$\\\hline $192$ & $41$\\\hline $7200$ &
$1681$\\\hline $4(\sqrt{g}+1)^2$ & $g=k^2, k\neq3, 5, 7, 11, 19,
41$\\\hline $4(g+1)$ & the remaining numbers\\\hline
\end{tabular}
\end{center}

\vspace{5pt}
\end{theorem}

\vskip 0.1true cm

\begin{theorem}\label{OEK}
The maximal orders $OE^k_g$ are given in the following table.

\vskip 0.1true cm

\begin{center}
\begin{tabular}{|c|c|}
\hline $OE^k_g$ & $g$\\\hline $12(g-1)$ & $9, 11, 121, 241$\\\hline
$2400$ & $361$\\\hline $6(g-1)$ & $2, 3, 4, 5, 21, 25, 97,
481$\\\hline $4(g-1)$ & the remaining numbers\\\hline
\end{tabular}
\end{center}
\end{theorem}

In fact, we will do something more. We will classify all the finite
group actions with order larger than $4(g-1)$. And the statements
above can be obtained directly from the following theorem.

\begin{theorem}\label{main table}
For an extendable finite group action $G$, if $|G|>4(g-1)$, all
possible relations between $|G|$ and $g$ are listed in the following
table. The foot index `$k$' means the action is realized only for a
knotted embedding, `$uk$' means the action can be realized for both
unknotted and knotted embeddings. If the action is realized only for
an unknotted embedding, there is no foot index.

\vskip 0.1true cm

\begin{tabular}{|c|c|}
\hline $|G|$ & $g$\\\hline $12(g-1)$ & $2, 3, 4, 5, 6, 9_{uk},
11_{uk}, 17 ,25, 97, 121_{uk}, 241_{uk}, 601$\\\hline $8(g-1)$ & $3,
7, 9, 49, 73$\\\hline $20(g-1)/3$ & $4, 16, 19, 361_{uk}$\\\hline
$6(g-1)$ I & $2, 3, 4, 5, 9_{uk}, 11, 17, 25, 97, 121_{uk},
241_{uk}$\\\hline $6(g-1)$ II & $\{2, 3, 4, 5, 9, 11, 25, 97, 121,
241\}_{uk}, 21_k, 481_k$\\\hline $24(g-1)/5$ & $6, 11, 41,
121$\\\hline $30(g-1)/7$ & $ 8, 29, 841, 1681$\\\hline
$4n(g-1)/(n-2)$ & $n-1, (n-1)^2$\\\hline
\end{tabular}

\vspace{5pt}
\end{theorem}

Here the $6(g-1)$ case contains two types, ``I" and ``II", we will
explain them in the next section.

Then some interesting phenomena appear: As expected, for all $g$
with finitely many exceptions we have $OE^u_g>OE^k_g$; indeed there
are only finitely many $g$ such that  $OE^u_g = OE^k_g$ and, a little bit
surprising, $OE^u_g < OE^k_g$ when $g = 21$ or $481$. Also for some
$g$, $OE^k_g = 12(g-1)$.

Our approach relies on the orbifold theory which is founded and
studied in \cite{Th}, \cite{Du1}, \cite{Du2}, \cite{BMP} and
\cite{MMZ}. More precisely, the proof of our main results translates
into the problem of finding the so-called {\it allowable}
2-orbifolds (Definition \ref{allowable}) in certain spherical
3-orbifolds. The strategy of such an approach will be given in
Section \ref{Strategy}.

In Section \ref{primary facts}, after introducing some basic notions
about orbifolds and finite group actions on manifolds, we present
sequences of observations concerning the orbifold pair $(S^3,
\Sigma_g)/G$ on both the topological level and the group theoretical
level which are very useful for our later approach. In Section
\ref{dunbar} we will describe Dunbar's list of spherical 3-orbifolds
whose underlying space is $S^3$. With the material prepared in
Sections \ref{primary facts} and \ref{dunbar}, we will be able to
explain why we can transfer the problem of finding $OE_g$ into the
problem of finding allowable 2-orbifolds in certain spherical
3-orbifolds  and,  more importantly, to outline how to get a
practical method  to find such 2-orbifolds. (Some people may prefer
just read the definitions in Section \ref{Strategy} and skip the
remaining part). In Section \ref{list 3-orbifolds} we will give the
list of 3-orbifolds containing allowable 2-suborbifolds which turns
out to be a small subset of Dunbar's list where the singular sets
are relatively simple. In Section \ref{2-orbifolds}, we will find
all allowable 2-orbifolds in the list of 3-orbifolds  provided by
Section \ref{list 3-orbifolds}, and then the main results are
derived. We end the paper by some examples.

\bigskip\noindent\textbf{Acknowledgement}. We thank the referee for his sensitivity on orbifold theory,
his aesthetic view on the structure of the paper, which enhance the paper significantly.
The first three authors are supported by grant No.11371034 of the
National Natural Science Foundation of China and Beijing
International Center for Mathematical Research.

%\bigskip\noindent\textbf{Acknowledgement}. The first author is supported by Beijing
%International Center for Mathematical Research�� Peking University.
%The second author is partially supported by grant No.10631060 of the
%National Natural Science Foundation of China.

\section{Orbifolds and finite group actions}\label{primary facts}

The orbifolds we consider have the form $M/H$, where $M$ is a
n-manifold and $H$ is a finite group acting faithfully and smoothly on
$M$. For each point $x\in M$, denote its stable subgroup by $St(x)$,
its image in $M/H$ by $x'$. If $|St(x)|>1$, $x'$ is called a
singular point and the singular index is $|St(x)|$, otherwise it is
called a regular point. If we forget the singular set we get a
topological space $|M/H|$ which is called underlying space. $M/H$ is
orientable if $M$ is orientable and $H$ preserves the orientation;
$M/H$ is connected if $|M/H|$ is connected.

We can also define covering spaces and fundamental group for an
orbifold. There is a one to one correspondence between orbifold
covering spaces and conjugacy classes of subgroup of the orbifold
fundamental group, and regular covering spaces correspond to normal
subgroups. A Van-Kampen theorem is also valid, see \cite{BMP}. In
the following, covering space or fundamental group refers always to
the orbifold setting.

\begin{definition}
A discal orbifold (spherical orbifold) has the form $B^n/H$
$(S^n/H)$, where $B^n$ $(S^n)$ is the $n$-dimension ball (sphere)
and $H$ a finite group acting orientation-preservingly on the
corresponding manifold. A handlebody orbifold has the form $V_g/H$.
\end{definition}

By a classical result for topological actions, $|B^2/H|$ is a disk,
possibly with one singular point. $SO(3)$ contains only five
classes of finite subgroups: the order $n$ cyclic group $C_n$, the
order $2n$ dihedral group $D_n$, the order 12 tetrahedral group $T$,
the order 24 octahedral group $O$, and the order 60 icosahedral
group $J$. It can be shown that $B^3/H$ ($S^2/H$) belongs to one of the
following five models. The underlying space $|B^3/H|$ ($|S^2/H|$) is
the 3- ball (the 2-sphere).

\begin{center}
\scalebox{0.6}{\includegraphics*[0pt,0pt][574pt,130pt]{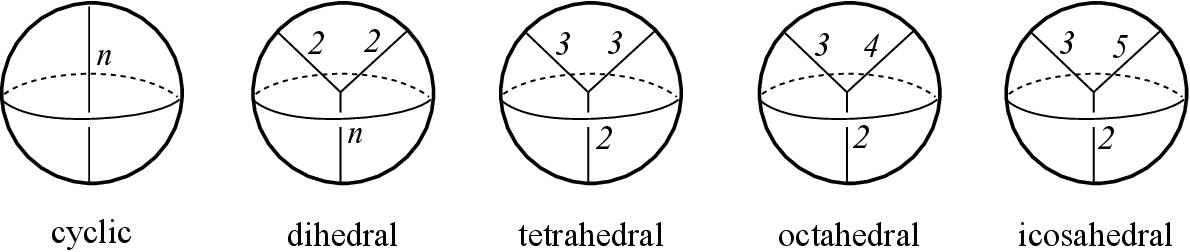}}

Figure 1
\end{center}

By the equivariant Dehn's lemma, see \cite{MY}, it can be shown that
a handlebody orbifold is the result of gluing finitely many 3-discal
orbifolds along some 2-discal orbifolds. And such gluing respecting
orientations always gives us a handlebody orbifold.

Like in the manifold case we can say that an orientable  seperating
2-suborbifold $\mathcal{F}$ in an orientable 3-orbifold
$\mathcal{O}$ is unknotted or knotted, depending on whether it
bounds handlebody orbifolds on both sides.

It is easy to see that if the underlying space of a handlebody
orbifold is a ball, then the singular set would form an unknotted
tree in the ball, possibly disconnected. Unknotted means that
the tree lies in a proper disk embedded in the ball. For more about handlebody orbifold theory one can see
\cite{MMZ}.

Suppose the action of $G$ on $\Sigma_g$ is extendable w.r.t. some
embedding $e: \Sigma_g\hookrightarrow S^3$; let $\tilde\Gamma=\{x\in
S^3 \mid \exists \,g\in G, g\neq id, s.t.\,gx=x\}$. As locally there
are only five kinds of model, $\tilde \Gamma$ is a graph, possibly
disconnected, and $S^3/G$ is a 3-orbifold whose singular set
$\Gamma=\tilde \Gamma/G$ is a trivalent graph. Each edge of $\Gamma$
can be labeled by an integer $n>1$ which indicates its singular
index. At each vertex the labels $m$, $q$, $r$ of the three adjacent
edges should satisfy $1/m+1/q+1/r > 1$. The 2-orbifold $\Sigma_g/G$
maps to the 2-suborbifold $e(\Sigma_g)/G$ whose singular set
$e(\Sigma_g)/G\cap \Gamma$ consists of isolated points.

We then have an orbifold covering  $p: S^3\rightarrow S^3/G$ and an
orbifold embedding $e/G: \Sigma_g/G\hookrightarrow S^3/G$.
Conversely, if we have an orbifold embedding from a 2-orbifold to a
spherical orbifold and the preimage of the 2-suborbifold in $S^3$ is
connected then we find an extendable action of $G$ on some surface
with respect to some embedding.

\begin{definition}
An orientable 2-suborbifold $\mathcal{F}$ in an orientable
3-orbifold $\mathcal{O}$ is compressible if either $\mathcal{F}$ is
spherical and bounds a discal 3-suborbifold in $\mathcal{O}$, or
there is a simple closed curve in $\mathcal{F}$ (not meeting the
singular set) which bounds a discal 2-orbifold in $\mathcal{O}$, but
does not bound a discal 2-orbifold in $\mathcal{F}$. Otherwise
$\mathcal{F}$ is called incompressible.
\end{definition}

\begin{lemma}\label{compressing}
Any orientable 2-suborbifold $\mathcal{F}$ in a spherical orbifold
$S^3/G$ is compressible.
\end{lemma}

\begin{proof}
$|\mathcal{F}|$ is two sided in $|S^3/G|$. Since $\pi_1(S^3/G)=G$ is
finite, $\pi_1(|S^3/G|)$ is also finite. Hence $\mathcal{F}$ cuts
$S^3/G$ into two parts $\mathcal{O}_1$, $\mathcal{O}_2$, and
$p^{-1}(\mathcal{F})$ cuts $S^3$ into several components $M_1, M_2,
\cdots, M_k$, each of which will be mapped by $p$ to one of the two
parts, the components have common boundary will be mapped to
different part.

If $\mathcal{F}$ is spherical, $p^{-1}(\mathcal{F})$ is a disjoint
union of 2-spheres. By the irreducibility of $S^3$ and $B^3$, one
$M_i$ must be a ball, hence one $\mathcal{O}_i$ is a discal
3-suborbifold and we have the result by definition.

Otherwise, $F = p^{-1}(\mathcal{F})$ is a disjoint union of
homeomorphic closed surfaces in $S^3$ of genus $g\ge 1$. Since $F$
is compressible in $S^3$ we can find an innermost compressing disk
$D$. Suppose $D$ is in $M_i$. By the equivariant Dehn's Lemma we can
find equivariant compressing disks in $M_i$. Suppose one of them is
$D'$, then all the images of $D'$ under the $G$ action will be
disjoint in $S^3$. Then it gives a `compressing disk' of
$\mathcal{F}$ in $S^3/G$.
\end{proof}

\begin{lemma}\label{4 singular pt}
Suppose $\mathcal{F}$ is a 2-suborbifold of a spherical orbifold
$S^3/G$ and $|\mathcal{F}|$ is homeomorphic to $S^2$.

(1) If $\mathcal{F}$ has not more than three singular points then
$\mathcal{F}$ is spherical and bounds a discal 3-orbifold.

(2) If $\mathcal{F}$ has precisely four singular points and the preimage of $\mathcal{F}$ in $S^3$ is connected, then
$\mathcal{F}$ bounds a handlebody orbifold in $S^3/G$.
\end{lemma}

\begin{proof}
As a 2-suborbifold, $\mathcal{F}$ should be spherical or has
`compressing disk' by Lemma \ref{compressing}.

(1) If $\mathcal{F}$ has no more than three singular points, every
simple closed curve in $\mathcal{F}$ bounds a discal orbifold in
$\mathcal{F}$. So $\mathcal{F}$ has no `compressing disk' and hence
is spherical, and then bounds a discal 3-orbifold.

(2) If $\mathcal{F}$ has precisely four singular points,
$\mathcal{F}$ is not spherical and hence has a `compressing disk'
$D$. Then $\partial D$ separates $\mathcal{F}$ into two discal
orbifolds $D_1, D_2$, each of which contains two singular points.
Now $D_1\cup D$ and $D_2\cup D$ are 2-suborbifolds in $S^3/G$ each
of which contains no more than three singular points; by the above
argument each of them bounds a discal 3-orbifold.

There are two cases: the discal 3-orbifold bounded by $D_1\cup D$
($D_2\cup D$) does not intersect the interior of $D_2$ ($D_1$). Then
the two discal 3-orbifolds meet only along $D$, and the result is
clearly a handlebody orbifold; otherwise for example the discal
3-orbifold, say $V$, bounded by $D_1\cup D$ intersects the interior
of $D_2$. Then $D_2$ is contained in $V$. Since the preimage of $\mathcal{F}$ in $S^3$ is connect, the complement of $V$ is also a discal 3-orbifold, and we can replace $V$ by its complement. Then we will meet the previous case.
\end{proof}

\begin{proposition}\label{bound handlebody}
Suppose $G$ acts on $(S^3, \Sigma_g)$. If $|G|
> 4(g-1)$, then $\Sigma_g/G$ has underlying space $S^2$ with four
singular points and bounds a handlebody orbifold, and $\Sigma_g$
bounds a handlebody.

In conclusion $OE_g\le OH_g\le 12(g-1)$.
\end{proposition}

\begin{proof}
$\Sigma_g/G$ is a 2-suborbifold in $S^3/G$ whose singular set
contains isolated points $a_{1},a_{2},\cdots,a_{k}$, with indices
$q_{1}\leq q_{2}\leq\cdots\leq q_{k}$. Note that $|S^3/G|$ and
$|\Sigma_g/G|$ are both manifolds. Suppose the genus of
$|\Sigma_g/G|$ is $\hat{g}$. By the Riemann-Hurwitz formula
$$2-2g=|G|(2-2\hat{g}-\sum_{i=1}^k (1-\frac{1}{q_i}))$$
we have
$$|G|=(2g-2)/(\sum_{i=1}^k (1-\frac{1}{q_i})+2\hat{g}-2).$$
If $\hat{g}\geq 1$ or $\hat{g}=0$, $k\geq 5$, then $|G|\leq 4g-4$.
Hence $\hat{g}=0$ and $k\leq 4$. If $k\leq 3$ then $\Sigma_g/G$
bounds a discal orbifold by Lemma \ref{4 singular pt} (1), which
leads to a contradiction (since $g>1$ by assumption). Hence $k=4$,
and by Lemma  \ref{4 singular pt} (2) $\Sigma_g/G$ bounds a
handlebody orbifold. In this case $\Sigma_g$ bounds a handlebody in
$S^3$.

By \cite{WWZZ}, or see Example \ref{cage}, $OE_g\ge 4(g+1)(>
4(g-1))$. Hence each $\Sigma_g$ in $S^3$ realizing $OE_g$ must
bounds a handlebody, therefore we have $OE_g\le OH_g\le 12(g-1)$.
\end{proof}

\begin{definition}
Let $\mathcal{F}$ be a 2-suborbifold in a spherical orbifold
$S^3/G$, with $|\mathcal{F}|$ homeomorphic to $S^2$ and has
precisely four singular points $a_1, a_2, a_3, a_4$. Supposing
$q_1\leq q_2\leq\ q_3\leq q_4$ for their indices, we call $(q_1,
q_2, q_3, q_4)$ the singular type of $\mathcal{F}$.
\end{definition}

\noindent Using the Riemann-Hurwitz formula, it is easy to see:

\begin{lemma}\label{singular type}
If $|G|>4(g-1)$ then the singular type of $\Sigma_g/G$ is one of
$(2, 2, 2, n)(n\ge 3)$, $(2, 2, 3, 3)$, $(2, 2, 3, 4)$, $(2, 2, 3,
5)$.
\end{lemma}

\begin{lemma}\label{compute genus}
The relation between the orders of extendable group actions and the
surface genus for a given singular type is given in the following
table:

\vspace{10pt}

\begin{tabular}{|c|c|c|c|c|}
\hline Type & $(2, 2, 2, n)(n\ge 3)$ & $(2, 2, 3, 3)$ & $(2, 2, 3,
4)$ & $(2, 2, 3, 5)$\\\hline Order &$4n(g-1)/(n-2)$ & $6(g-1)$ &
$24(g-1)/5$ & $30(g-1)/7$\\\hline
\end{tabular}

\vspace{10pt}

\end{lemma}

\begin{lemma}\label{handlebody shape}
If the singular type of $\Sigma_g/G$ is not $(2, 2, 3, 3)$, the
handlebody orbifold bounded by $\Sigma_g/G$ is as in Figure 2(a); if
the singular type is $(2, 2, 3, 3)$, there are the two possibilities
in Figure 2(a) and (b) for this handlebody orbifold.
\end{lemma}

\begin{center}
\scalebox{0.5}{\includegraphics*[0pt,0pt][668pt,212pt]{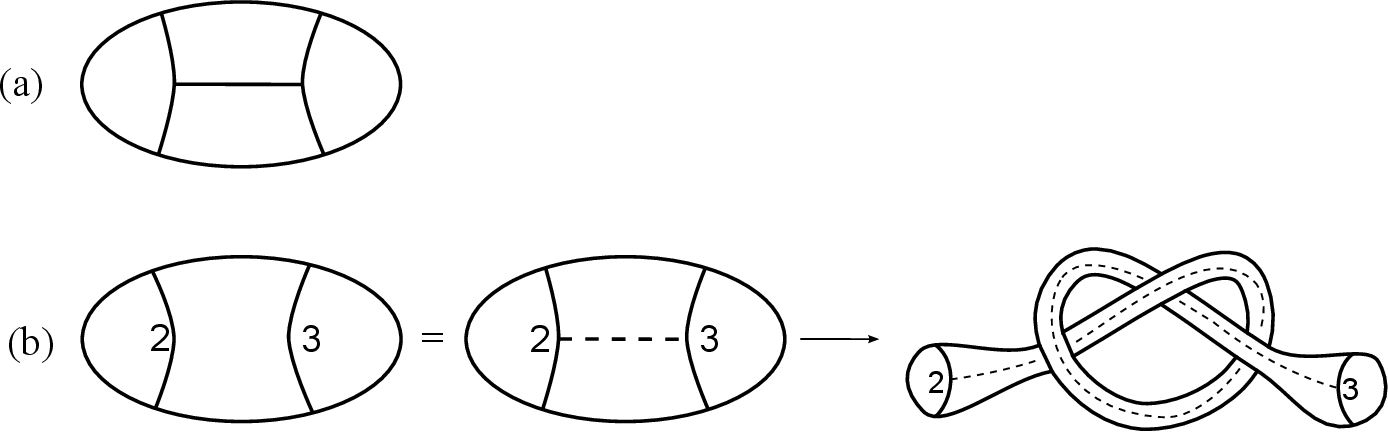}}

Figure 2
\end{center}

\begin{proof}
By the proof of Lemma \ref{4 singular pt}, the handlebody orbifold
bounded by $\Sigma_g/G$ has underlying space $B^3$ and singular set
a tree like in Figure 2(a) or two arcs. The indices of the end
points of an arc must be the same. Hence if the singular set
contains two arcs, the singular type must be $(2, 2, 3, 3)$.
\end{proof}

Note that in the case of Figure 2(a) the handlebody orbifold is a
regular neighborhood of a singular edge. In the case of Figure 2(b)
the handlebody orbifold is a regular neighborhood of a dashed arc,
and this dashed arc can be locally knotted as in the figure. We can
also say that a singular edge/dashed arc is unknotted or knotted,
depending on whether the boundary of its regular neighborhood is
unknotted or knotted.

\begin{lemma}\label{connected}
Suppose a finite group $G$ acts on $(M, F)$, where $M$ is a
3-manifold, with a surface embedding $i: F\hookrightarrow M$, so we
have diagrams:
$$\xymatrix{
  F \ar[d]_{p} \ar[r]^{i} & M \ar[d]^{p}
  & \pi_1(F) \ar[d]_{p_*} \ar[r]^{i_*} & \pi_1(M) \ar[d]^{p_*} \\
  F/G  \ar[r]^{\hat{i}} & M/G
  & \pi_1(F/G)  \ar[r]^{\hat{i}_*} & \pi_1(M/G)
  }
$$
Suppose  $F/G$ is connected. Then $F$ is connected if and only if
$$\hat{i}_*(\pi_1(F/G))\cdot p_*(\pi_1(M))=\pi_1(M/G).$$
\end{lemma}

\begin{proof}
$"\Longrightarrow"$ Suppose $\hat{i}_*(\pi_1(F/G))\cdot
p_*(\pi_1(M))\subsetneqq\pi_1(M/G)$. We can find an orbifold
covering space $\widehat{M}$ corresponds to
$\hat{i}_*(\pi_1(F/G))\cdot p_*(\pi_1(M))$. Then we have diagram:
$$\xymatrix{
  F \ar[2,0]_{p} \ar[0,2]^{i}
                & & M \ar[dl] \ar[2,0]^{p} \\
               & \widehat{M} \ar[dr]^{\hat{p}} &  \\
  F/G \ar@{-->}[ur] \ar[0,2]^{\hat{i}}
               &  & M/G             }
$$
Since $\hat{i}_*(\pi_1(F/G))\subseteq
\hat{p}_*(\pi_1(\widehat{M}))$, $F/G$ can lift to $\widehat{M}$, and
it lifts to a disjoint union of copies. Hence $F$ must be
disconnected.

$"\Longleftarrow"$ Suppose  $F$ is not connected.  Let
$F_1\subsetneqq F$ be a component of $F$ and $G_1$ its stablilizer
in $G$, that is $G_1=\{ h\in G \mid h(F_1)=F_1\}$. Then
$F_1/G_1=F/G$. Now $|\pi_1(M/G):p_*(\pi_1(M))|=|G|$, and

%we will compute$|\hat{i}_*(\pi_1(F/G))\cdot p_*(\pi_1(M)): p_*(\pi_1(M))|$
\begin{eqnarray*}
&&|\hat{i}_*(\pi_1(F/G))\cdot p_*(\pi_1(M)): p_*(\pi_1(M))| \\
&= &|\hat{i}_*(\pi_1(F/G))\cdot p_*(\pi_1(M))/ p_*(\pi_1(M))|\\
&= &|\hat{i}_*(\pi_1(F/G))/\hat{i}_*(\pi_1(F/G))\cap p_*(\pi_1(M))|\\
&\leq &|\hat{i}_*(\pi_1(F/G)):\hat{i}_*p_*(\pi_1(F_1))|\\
&= &|\pi_1(F/G)/ker\hat{i}_*:p_*(\pi_1(F_1))\cdot ker\hat{i}_*/ker\hat{i}_*|\\
&= &|\pi_1(F_1/G_1):p_*(\pi_1(F_1))\cdot ker\hat{i}_*|\\
&\leq &|\pi_1(F_1/G_1):p_*(\pi_1(F_1))|\\
&= &|G_1|< |G|.
\end{eqnarray*}
Hence $\hat{i}_*(\pi_1(F/G))\cdot p_*(\pi_1(M))\subsetneqq
\pi_1(M/G)$.
\end{proof}

\begin{remark}\label{surjective} (1) When F is connected, we have
$\hat{i}_*p_*(\pi_1(F))=\hat{i}_*(\pi_1(F/G))\cap p_*(\pi_1(M))$ and
$ker(\hat{i}_*) \subseteq p_*(\pi_1(F))$. If $M$ is simply
connected, then $F$ is connected if and only if $\hat{i}_*$ is
surjective.

(2) If $F/G\subset S^3/G$ bounds handlebody orbifolds on both sides
then clearly $\hat{i}_*$ is surjective.
\end{remark}

\begin{corollary}\label{connect surface}
Suppose $\mathcal{F}$ is a connected 2-suborbifold with an embedding
$\hat{i}: \mathcal{F}\hookrightarrow S^3/G$ into a spherical
orbifold $S^3/G$. Let $p^{-1}(\mathcal{F})=\Sigma$; then $\Sigma$ is
connected if and only if $\hat{i}_*$ is surjective.
\end{corollary}

\begin{lemma}[Edge killing]\label{edge kill}
Let $(X, \Gamma)$ be an orientable 3-orbifold with underlying space
$X$ and singular set a trivalent graph $\Gamma$. An edge killing
operation is defined as in Figure 3, where $(a,b)$ denotes the
greatest common divisor of $a$ and $b$. If via an edge killing
operation we get from $\Gamma$ a new graph $\Gamma'$, then we have a
surjective homomorphism $\pi_1(X, \Gamma)\rightarrow \pi_1(X,
\Gamma')$.
\end{lemma}

\begin{center}
\scalebox{0.65}{\includegraphics*[0pt,0pt][546pt,77pt]{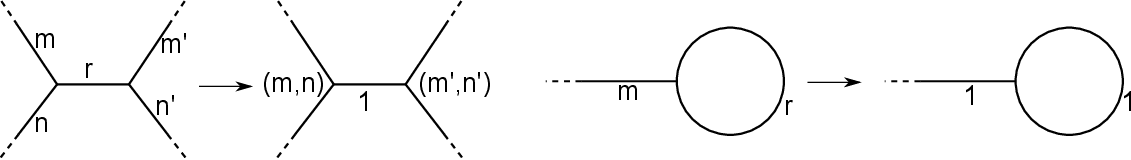}}

Figure 3
\end{center}

\begin{proof}
Denoting by $N(\Gamma)$ a regular neighborhood of $\Gamma$ in $X$,
there is a surjective homomorphism from $\pi_1(X-N(\Gamma))$ to
$\pi_1(X, \Gamma)$, and we can compute $\pi_1(X, \Gamma)$ from
$\pi(X-N(\Gamma))$ by adding relations like $x^r=1$ \cite{BMP}. The
effect of an edge killing operation on fundamental groups is just
adding relations like $x=1$, and then we obtain a presentation of
$\pi(X, \Gamma')$.
\end{proof}

\begin{remark}
This edge killing operation is just a way to get a quotient group.
Using it we can show some $\hat{i}_*$ is not surjective. The
orbifold $(X, \Gamma')$ may be not a good one (be covered by a
manifold).
\end{remark}

\begin{lemma}\label{2sphere}
Let $G$ be an extendable finite group action with respect to some
embedding $e: \Sigma_g\hookrightarrow S^3$. If $|e(\Sigma_g)/G|$ is
homeomorphic to $S^2$, then $|S^3/G|$ is homeomorphic to $S^3$.
\end{lemma}

\begin{proof}
By Corollary \ref{connect surface} the homomorphism $(e/G)_*:
\pi_1(\Sigma_g/G)\rightarrow \pi_1(S^3/G)$ is surjective. By Lemma
\ref{edge kill}, if we kill all the singular edges we get a
surjection $\pi_1(|\Sigma_g/G|)\rightarrow \pi_1(|S^3/G|)$. Hence
$\pi_1(|S^3/G|)$ is trivial and $|S^3/G|$ is homeomorphic to $S^3$.
\end{proof}

\section{Dunbar's list of spherical 3-orbifolds}\label{dunbar}
In \cite{Du1}, \cite{Du2} Dunbar lists all spherical
orbifolds with underlying space $S^3$. We list these pictures below
and give a brief explanation such that one can check graphs
conveniently. For more information, one should see the original
papers.

Since the underlying space is $S^3$, all the information is
contained in the trivalent graphs of the singular sets. Each edge in
a graph is labeled by an integer indicating the singular index of
the edge, with the convention that each unlabeled edge has index 2.
If a graph has a vertex such that the incident edges have labels
$(2, 3, 3)$, $(2, 3, 4)$ or $(2, 3, 5)$, the orbifold is non-fibred.
All the non-fibred spherical orbifolds have underlying space $S^3$
and are listed in Table III. Otherwise the orbifolds are Seifert
fibred and are listed in Table I (the basis of the fibration is a
2-sphere) and Table II (the basis is a disk).

\vspace{20pt}

\begin{center}
Table I: Fibred case with base $S^2$

\vspace{10pt}

\scalebox{0.4}{\includegraphics*[0pt,0pt][392pt,162pt]{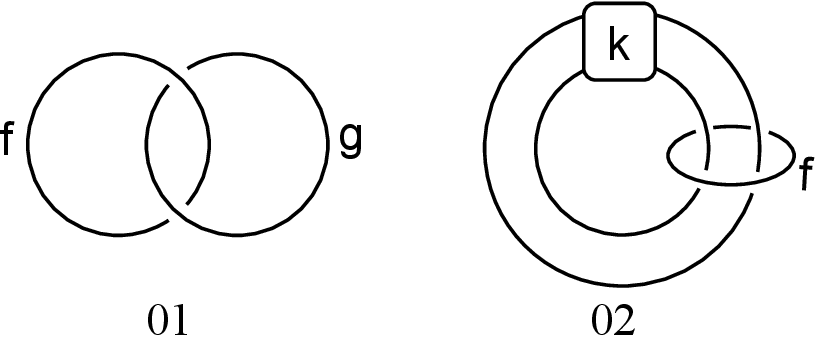}}

\scalebox{0.4}{\includegraphics*[0pt,0pt][766pt,173pt]{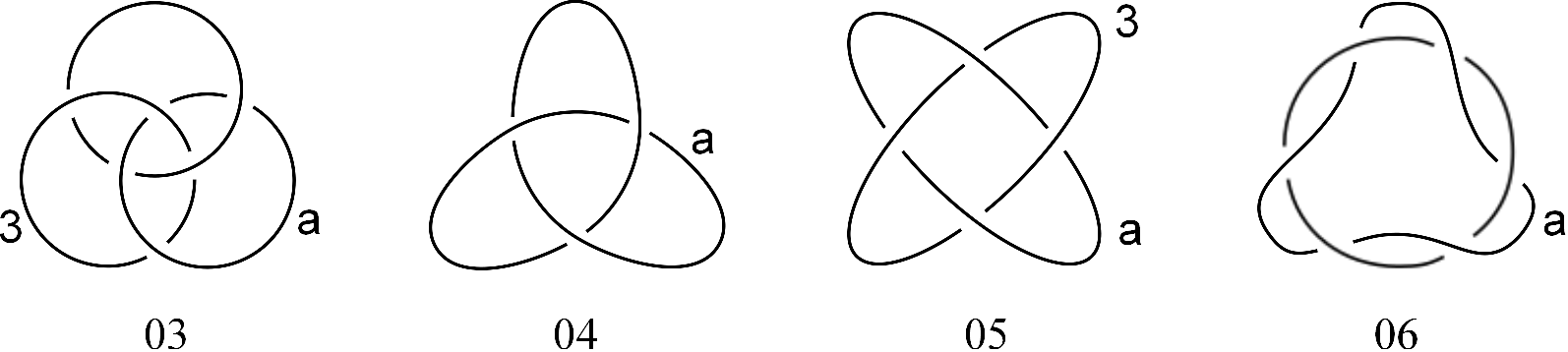}}

\scalebox{0.4}{\includegraphics*[0pt,0pt][766pt,176pt]{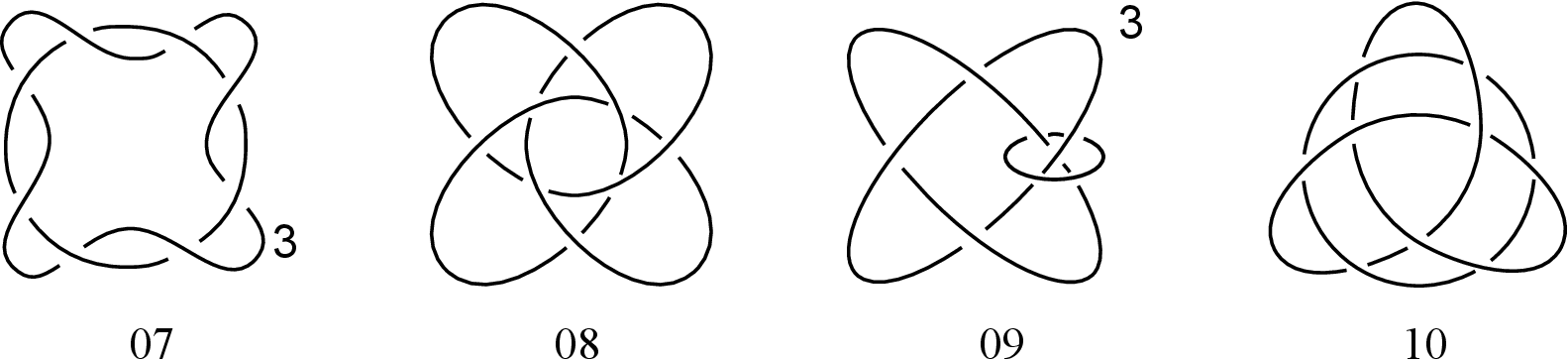}}

\scalebox{0.4}{\includegraphics*[0pt,0pt][566pt,187pt]{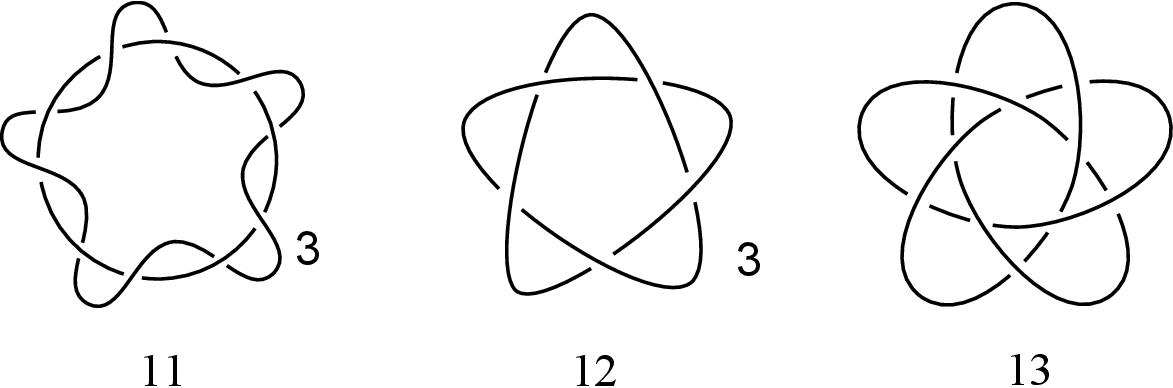}}

\end{center}

\vspace{50pt}

\begin{center}
Table II: Fibred case with base $D^2$

\vspace{10pt}

\scalebox{0.5}{\includegraphics*[0pt,0pt][680pt,950pt]{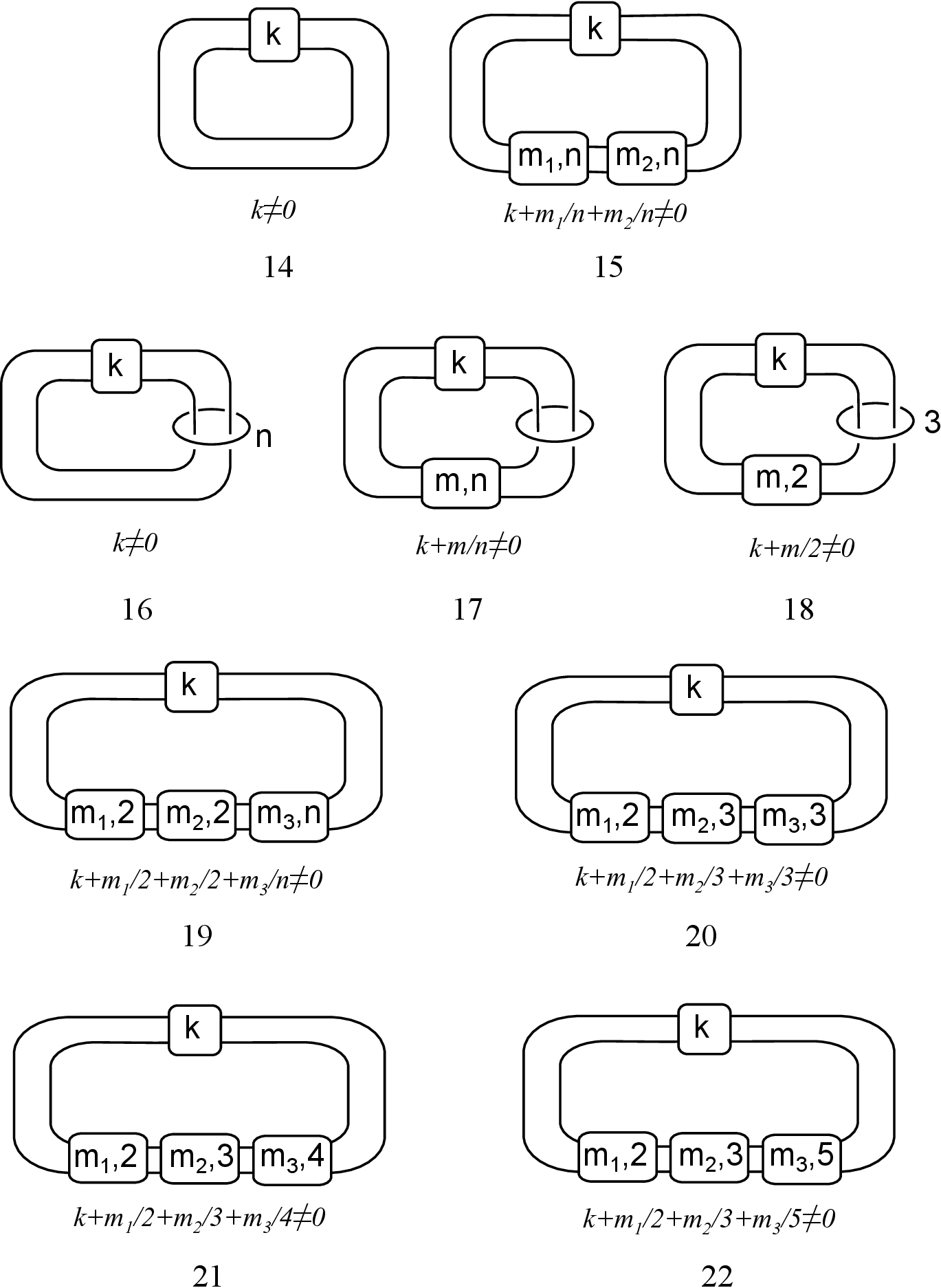}}

\end{center}

In Table I and Table II many graphs have some free or undetermined
parameters (just called parameters in the following). These
parameters should satisfy $n>1$, $3\le a\le5$, $f\ge 1$, $g\ge 1$,
and in Table I we require $k\neq 0$. The letter `@' means
amphicheiral (there exists an orientation-reversing homeomorphism of
the orbifold). If an orbifold is non-amphicheiral, as in the
original paper its mirror image is not presented.

A box with an integer $k$ indicates two parallel arcs with $k$-half
twists, the over-crossings from lower left to upper right if $k>0$,
and upper left to lower right if $k<0$. A box with two integer $m,
n$ stands for a picture as in Figure 4 and Figure 5; It satisfies
$|2m|\le n$.

\vspace{20pt}

\begin{center}

\scalebox{0.5}{\includegraphics*[0pt,0pt][340pt,60pt]{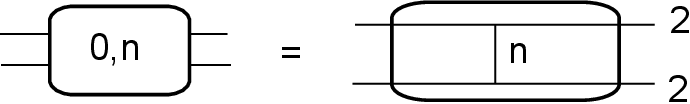}}

Figure 4
\end{center}

\vspace{40pt}

\begin{center}
\scalebox{0.7}{\includegraphics*[0pt,0pt][333pt,220pt]{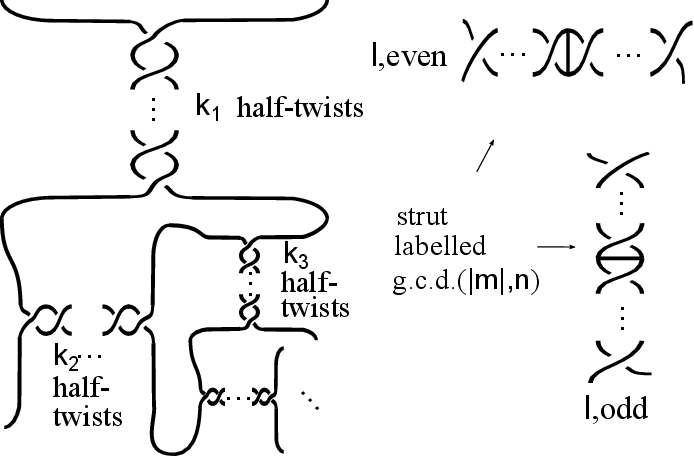}}
%\hspace*{\fill}
\raisebox{80pt}{\scalebox{1}{\parbox{120pt}{$\cfrac{|m|}{n}=
\cfrac{1}{k_1+\cfrac{1}{k_2+\cfrac{1}{\ddots+\cfrac{1}{k_l}}}}$}}}

Figure 5
\end{center}

\vspace{20pt}

All the crossing numbers of the horizontal and vertical parts are
determined by the unique continued fraction presentation of $|m|/n$,
such that all $k_i$ are positive and $k_l\ge 2$. All the
over-crossings are from lower left to upper right if $m>0$, and from
upper left to lower right if $m<0$. If the greatest common divisor
$(|m|, n)=d>1$, we add a `strut' labeled $d$ in the $k_l$ twist as
shown in the picture. If $m=0$, we add a `strut' labeled $n$ between
two parallel lines.

\newpage

\begin{center}

Table III: Non-fibred case

\vspace{20pt}
\scalebox{0.5}{\includegraphics*[0pt,0pt][660pt,1130pt]{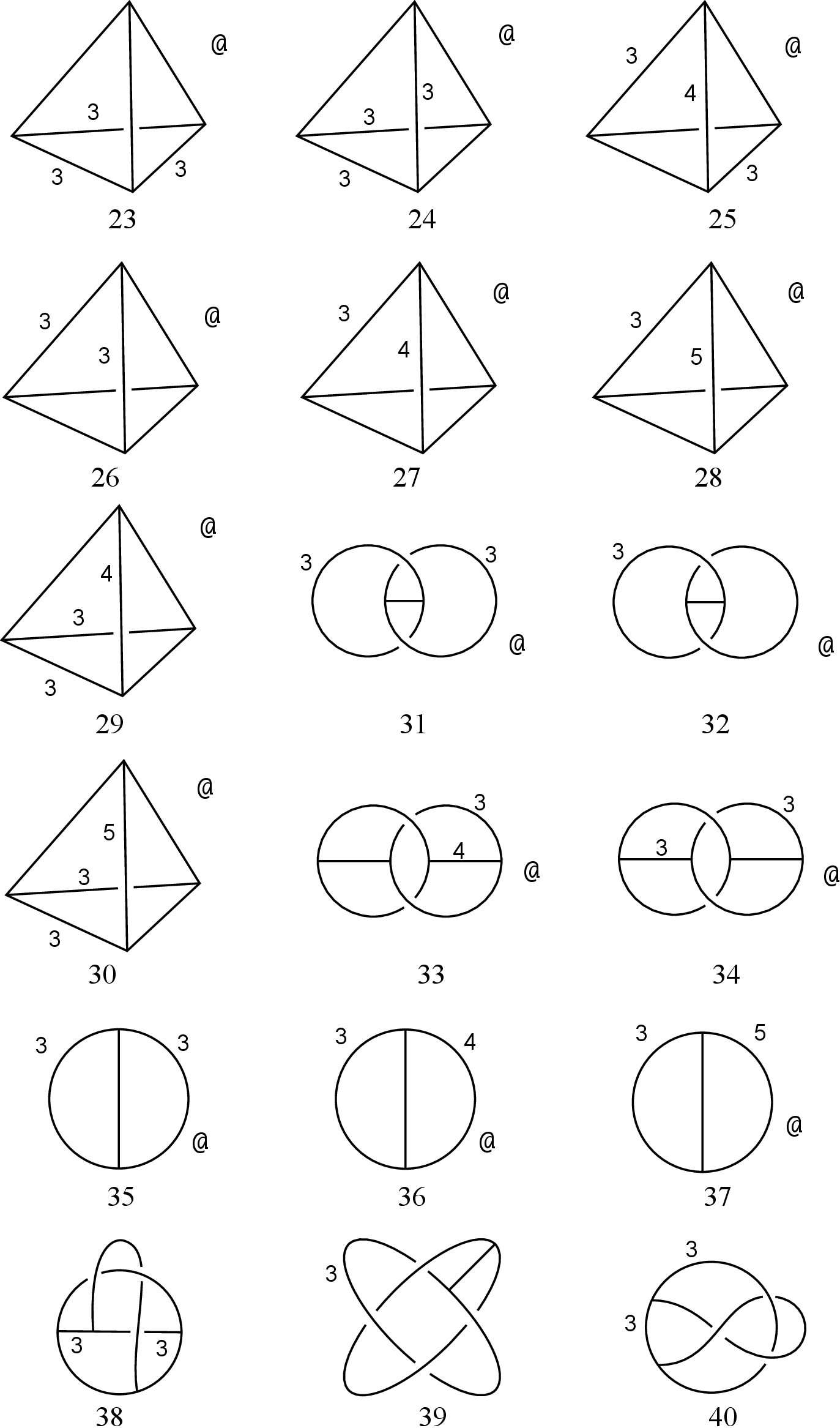}}

\end{center}

\newpage

\section{Strategy and outline of finding $OE_g$}\label{Strategy}
{\bf 1. Obtain $OE_g$ from allowable 2-suborbifolds in spherical
3-orbifolds.}

We already know $OE_g\ge 4(g+1)$ by Example \ref{cage}, or see
\cite{WWZZ}, hence to determine $OE_g$ we can assume $|G|>4(g-1)$.

\begin{definition}\label{allowable}
A 2-suborbifold $\mathcal{F}$ in a spherical 3-orbifold $S^3/G$,
with $|G|>4(g-1)$, is called allowable if its preimage in $S^3$ is a
closed connected surface $\Sigma_g$. A singular edge/dashed arc is
called allowable if the boundary of its regular neighborhood is
allowable.
\end{definition}

Therefore if $G$ acts on $(S^3, \Sigma_g)$ and realizes $OE_g$ then
$\mathcal{F}=\Sigma_g/G\subset S^3/G$ must be an allowable
2-suborbifold. We intend to find extendable actions from allowable
2-suborbifolds in spherical 3-orbifolds and, more weakly, to find
the maximum orders of extendable actions from certain information
about such allowable 2-suborbifolds.

Suppose we have a spherical 3-orbifold $S^3/G$ and an allowable
2-suborbifold $\mathcal{F}\subset S^3/G$. By Proposition \ref{bound
handlebody}, $\mathcal{F}$ has underlying space $S^2$ with four
singular points, and moreover $\mathcal{F}$ has a singular type as
in the list of Lemma \ref{singular type}. Once we know the singular
type of $\mathcal{F}$ and the order of $G$, we know the genus of the
corresponding closed connected surface $\Sigma_g\subset S^3$ such
that $(S^3, \Sigma_g)/G=(S^3/G, \mathcal{F})$ by Lemma \ref{compute
genus}. So if we know the singular types of all allowable
2-orbifolds in $S^3/G$, then we know all $\Sigma_g$ which admit an
extendable action of the group $G$ with $|G|>4(g-1)$; in other
words, for a fixed $g$ we know if $\Sigma_g$ admits an extendable
actions of group $G$ with $|G|>4(g-1)$. Hence if we know the
singular types of all allowable 2-orbifolds in all spherical
3-orbifolds $S^3/G$, then for a fixed $g$ we know all finite groups
$G\subset SO(4)$ such that $\Sigma_g$ admits an extendable action of
the group $G$ with $|G|>4(g-1)$, and consequently $OE_g$ can be
determined.

\noindent {\bf 2. List all allowable 2-suborbifolds in spherical
3-orbifolds.}

\begin{definition}
A 2-sphere in a spherical 3-orbifold $S^3/G$ is called candidacy if
it intersects the singular graph of $S^3/G$ in exactly four singular
points of one of the types listed in Lemma \ref{singular type}. A
singular edge/dashed arc is called candidacy if the underlying space
of the boundary of its regular neighborhood is candidacy.
\end{definition}

Clearly for each allowable 2-suborbifold $\mathcal{F}\subset S^3/G$,
$|\mathcal{F}|\subset S^3/G$ is a candidacy 2-sphere. On the other
hand, each candidacy 2-sphere is the underlying space of a
non-spherical  2-orbifold $\mathcal{F}$, and we will denote this
candidacy 2-sphere by $|\mathcal{F}|$.

We say that a 2-orbifold $\hat{i}: \mathcal{F}\subset S^3/G$ is
$\pi_1$-surjective if the induced map on the orbifold fundamental
groups is surjective.

The process of  listing all allowable 2-suborbifolds in spherical
3-orbifolds is  divided into two steps:

(i) List all spherical 3-orbifolds containing allowable
2-suborbifolds.

Suppose $\hat{i}: \mathcal{F}\subset S^3/G$ is an allowable
2-suborbifold in a spherical 3-orbifold. Then the preimage of
$\mathcal{F}$ must be connected, and by Lemma \ref{2sphere} the
underlying space of $S^3/G$ is $S^3$. All  spherical 3-orbifolds
with underlying space $S^3$ are listed in Dunbar's lists provided in
Section \ref{dunbar}.  Below we will denote spherical 3-orbifolds
with underline space $S^3$ by $(S^3, \Gamma)$, where $\Gamma$ is the
singular set.

Since $\hat{i}: \mathcal{F}\subset (S^3, \Gamma)$ is  allowable, the
preimage of $\mathcal{F}$ is connected, and by Corollary
\ref{connect surface},  $\hat{i}$ is $\pi_1$-surjective. Let $(S^3,
\Gamma)$ be a spherical 3-orbifold with parameters; we will show
that, for each 2-suborbifold $\hat{i}: \mathcal{F}\hookrightarrow
(S^3,\Gamma)$ such that $|\mathcal{F}|$ is a candidacy 2-sphere and
$\hat{i}$ is $\pi_1$-surjective, the parameters must satisfy certain
equations. Then we can determine the parameters and get a list of
spherical 3-orbifold containing allowable 2-suborbifolds which is a
small subset of Dunbar's list where the singular sets are relatively
simple. Step (i) will be carried in Section \ref{list 3-orbifolds}.

(ii) List all allowable 2-suborbifolds in each spherical 3-orbifold
obtained in Step (i).

How to find such 2-suborbifolds? Indeed this is already the question
we must face in Step (i). Precisely, this question divides into two
subquestions:

(a) How to find candidacy  2-spheres $|\mathcal{F}|$ in a given
spherical 3-orbifold $(S^3, \Gamma)$?

(b) For each candidacy 2-sphere $|\mathcal{F}|$  we find, how to
verify if $\hat{i}$ is $\pi_1$-surjective?

A simple and crucial fact in solving Question (a) is provided by
Proposition \ref{bound handlebody}: For each candidacy 2-sphere
$|\mathcal{F}|$, the 2-orbifold $\mathcal{F}$ must bound a
handlebody orbifold $V$; moreover the shape of $V$ is given in Lemma
\ref{handlebody shape}.

If the singular type is not $(2, 2, 3, 3)$, then $V$ is a regular
neighborhood of a singular edge. In this case we can check all the
edges to see whether the corresponding singular type is contained in
the list of Lemma \ref{singular type}.

\begin{center}
\scalebox{0.5}{\includegraphics*[0pt,0pt][650pt,90pt]{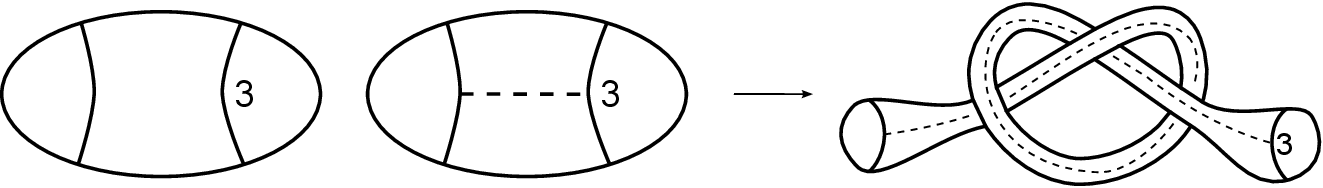}}

Figure 6
\end{center}

If the singular type is $(2, 2, 3, 3)$, there are two possibilities
for the shape of $V$. The new one can be thought of as a
neighborhood of a regular arc with its two ends on singular edges
labeled  $2$ and $3$ which will be presented  by a dashed arc. If
there is such a dashed arc then we can locally knot this arc in an
arbitrary way and obtain infinitely many candidacy 2-spheres, see
Figure 6; in this case we only give one such dashed arc, and this
will be the unknotted one if it exists.

For Question (b): If $\mathcal{F}\subset (S^3,\Gamma)$ bounds
handlebody orbifolds on both sides then  $\hat{i}_*$ is surjective
(Remark \ref{surjective} (2)). To verify the $\pi_1$-surjectivity of
$\hat{i}: \mathcal{F}\subset (S^3, \Gamma)$ for the knotted cases,
we are still lucky, all $\pi_1$-surjective cases  of $\hat{i}:
\mathcal{F}\subset (S^3, \Gamma)$ can be verified by the so-called
coset enumeration method, and all non-$\pi_1$-surjective cases can
be verified by the edge killing method of Lemma \ref{edge kill},
with three exceptions where Lemma \ref{surjection} will be applied.

\section{List of fibred 3-orbifolds containing allowable
2-suborbifolds.}\label{list 3-orbifolds}

We will establish the equations (and inequalities) which the
parameters in Dunbar's list must satisfy in order to contain
allowable 2-suborbifolds. We will solve these equations to get all
solutions and redrew the pictures of the corresponding 3-orbifolds.
Since different solutions often give the same orbifold up to the
automorphisms of the orbifold, we will only draw the graphs of
non-homeomorphic orbifolds.

Note first all graphs having parameters are contained in Table I and
Table II.

Suppose $|\mathcal{F}|\subset |(S^3, \Gamma)|$ is a candidacy
2-sphere. Then $\mathcal{F}\subset (S^3, \Gamma)$ bounds a
handlebody orbifold $V$ of given singular type by the discussion in
last section.

To determine the parameters, we divide the discussion into two
cases:

Case 1. The singular type is not $(2, 2, 3, 3)$; then $V$ is as in
Figure 2(a).

Case 2. The singular type is $(2, 2, 3, 3)$, and $V$ is as in Figure
2(a) or as in Figure 2(b).

In Case 1, $\Gamma\cap V$ has two degree $3$ vertices, and
$\mathcal{F}=\partial V$ has  singular type $(2, 2, 2, n), n\ge 3$,
$(2,2,3,4)$ or $(2,2,3,5)$. Hence $\Gamma\cap V$ must be a label $2$
arc adding two `strut segments' with different labels, see Figure 7;
$(r,s)$ is either $(2,n)$, or $(3,4)$, or $(3,5)$.

Only graphs 15, 19, 20, 21, 22 in Table II have more than one strut.
So we need only to deal with these five graphs in Case 1.

\begin{center}
\scalebox{0.6}{\includegraphics*[0pt,0pt][385pt,82pt]{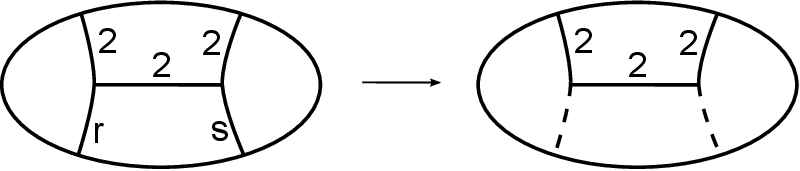}}

Figure 7
\end{center}

The way to determine the parameters was suggested in the last
section:

Suppose $\hat{i}: \mathcal{F}\subset (S^3, \Gamma)$ is
$\pi_1$-surjective. If  we kill these two `strut segments', we
obtain $\hat{i}: \mathcal{F'}\subset (S^3, \Gamma')$ which is also
$\pi_1$-surjective. Since $\pi_1(\mathcal{F}')=\Z_2$, it follows
$|\pi_1(S^3, \Gamma')|\le 2$, therefore  $\Gamma'$ contains no other
`strut' (otherwise $\pi_1((S^3, \Gamma'))$ would not be cyclic), and
hence $\Gamma'$ is a Montesinos link labeled by $2$. The double
branched cover of $S^3$ over $\Gamma'$ must be a 3-manifold $N$ with
trivial $\pi_1(N)$ (and hence $N=S^3$, that is to say $\Gamma'$ is a
trivial knot by the positive solution of the Smith conjecture). We
use the parameters to compute $\pi_1(N)$ and then determine the
parameters.

\begin{center}
\scalebox{0.6}{\includegraphics*[0pt,0pt][263pt,128pt]{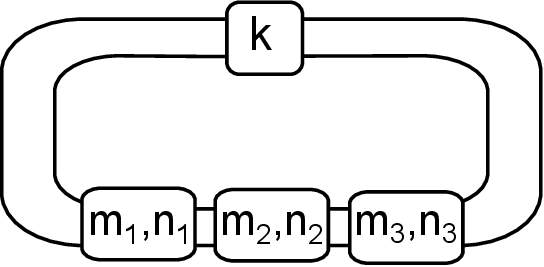}}

Figure 8
\end{center}

For short we present the graphs 15, 19, 20, 21, 22 by a single graph
in Figure 8; the five graphs correspond to the choices $(n, n, 1)$,
$(2, 2, n)$, $(2, 3, 3)$, $(2, 3, 4)$ and $(2, 3, 5)$ for $(n_1,
n_2, n_3)$ ($n>1$). Since $|2m_i|\le n_i$ and $\Gamma$ contains
exactly two `struts' with different labels, we have condition

$|2m_i|\le n_i$, at least one $m_i$ is zero and at least one $m_i$
is nonzero. $(*)$

Let $(|m_i|, n_i)=d_i$ be the greatest common divisor of $|m_i|$ and
$n_i$, by the singular type restrictions we have

$\{d_1, d_2, d_3\}=\{1, 2, d\}, (d>2)$, or $\{1, 3, 4\}$ or $\{1, 3,
5\}$. $(**)$

Write $m_i=m_i'd_i$, $n_i=n_i'd_i$, then $\Gamma'$ is the Montesinos
link presented by Figure 8, with each $(m_i, n_i)$ replaced by
$(m'_i, n'_i)$. By a theorem of Montesinos \cite[Proposition
12.30]{BZ}, the double branched cover of $S^3$ over $\Gamma'$ is a
Seifert manifold $N$ whose fundamental group has the following
presentation:

$$\pi_1(N)=\langle x, y, z, t \mid x^{n_1'}t^{m_1'}=y^{n_2'}t^{m_2'}=z^{n_3'}t^{m_3'}=1,$$
$$xyzt^{-k}=1, [x, t] = [y, t] = [z, t] = 1\rangle$$

If $m_i=0$ for some $i$, then $n_i'=1$ by definition, and it is easy
to see $\pi_1(N)$ is an abelian group. Now $\pi_1(N)$ is trivial if
and only if the determinant of the presentation matrix is $\pm 1$.
Hence we have
$$kn_1'n_2'n_3'+m_1'n_2'n_3'+n_1'm_2'n_3'+n_1'n_2'm_3'=1. \qquad (1)$$
or
$$kn_1'n_2'n_3'+m_1'n_2'n_3'+n_1'm_2'n_3'+n_1'n_2'm_3'=-1. \qquad(1)'$$
Dividing $n_1'n_2'n_3'$ on both sides and using the facts that
$|2m_i'|\le n_i'$ and $m_i'=0$ for some $i$, we have
$$|k|\le 1/n_1'n_2'n_3'+|m_1'|/n_1'+|m_2'|/n_2'+|m_3'|/n_3'<2$$
and then $k=0, \pm 1$.

Noticing that solutions of $(1)$ and $(1)'$ have a one to one
correspondence if we change the signs of $k$ and $m_i'$
simultaneously. And the corresponding knot of one solution is the
mirror image of the other. Hence we need only deal with $(1)$.

In
$(1)$ since $|2m_i'|\le n_i'$ and one $m_i'$ is zero, $k$ can not be
$-1$. Otherwise, for example assume $m_1'=0$ and $k=-1$, then $n_1'=1$ and
$(1)$ becomes
$$-n_2'n_3'+m_2'n_3'+n_2'm_3'=1.$$
But since $|2m_i'|\le n_i'$,
$$-2n_2'n_3'+2m_2'n_3'+2n_2'm_3'\le
-2n_2'n_3'+n_2'n_3'+n_2'n_3'=0.$$

Now for a given choice of $(n_1, n_2, n_3)$, We will find all
possible solutions $(k, m_1, m_2, m_3)$ or $(k, m_1, m_2, m_3, n)$
satisfying $(*), (**)$ and $(1)$ (or $(1)'$). Changing signs of a
solution will give us a mirror image. Moreover a picture of a
solution with $k=1$ always isomorphic to a picture of solution with
$k=0$ (an illustration is given in Figure 10, the illustrations of
remaining cases are similar). We only draw the graphs of the
non-homeomorphic orbifolds:

When $(n_1, n_2, n_3)=(2, 3, 3), (2, 3, 4), (2, 3, 5)$, it is easy
since we can enumerate the possibilities.

\noindent If $(n_1, n_2, n_3)=(2, 3, 3)$, then $(k, m_1, m_2,
m_3)=(0, 0, \pm1, 0)$, $(0, 0, 0, \pm1)$. We present the picture for
the case of $(0,0, 0, 1)$.

\begin{center}
\scalebox{0.5}{\includegraphics*[0pt,0pt][207pt,114pt]{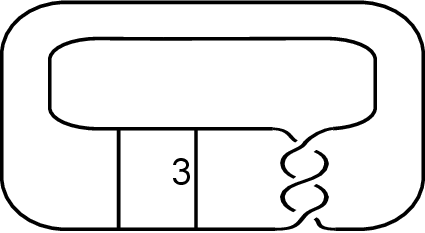}}

Figure 9
\end{center}

\noindent If $(n_1, n_2, n_3)=(2, 3, 4)$, then $(k, m_1, m_2,
m_3)=(0, 0, 0, \pm1)$, $(0, 0, \pm1, 0)$,  $(0, \pm1, 0, 0)$ or
$(\pm1, \mp1, 0, 0)$. We present $(0, 0, 0, 1)$, $(0, 0, 1, 0)$,
$(0, 1, 0, 0)$ and $(1, -1, 0, 0)$. It is clear the last two
pictures are homeomorphic.

\begin{center}
\scalebox{0.5}{\includegraphics*[0pt,0pt][465pt,264pt]{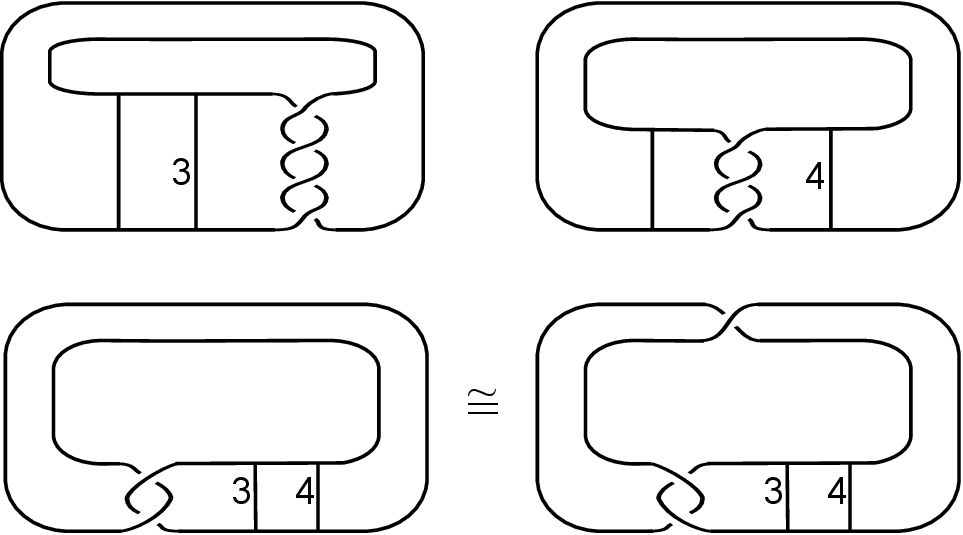}}

Figure 10
\end{center}

\noindent If $(n_1, n_2, n_3)=(2, 3, 5)$, then $(k, m_1, m_2,
m_3)=(0, 0, 0, \pm1)$,  $(0, 0, \pm1, 0)$,  $(0, \pm1, 0, 0)$ or
$(\pm1, \mp1, 0, 0)$. We present $(0, 0, 0, 1)$, $(0, 0, 1, 0)$ and
$(0, 1, 0, 0)$.

\begin{center}
\scalebox{0.5}{\includegraphics*[0pt,0pt][685pt,114pt]{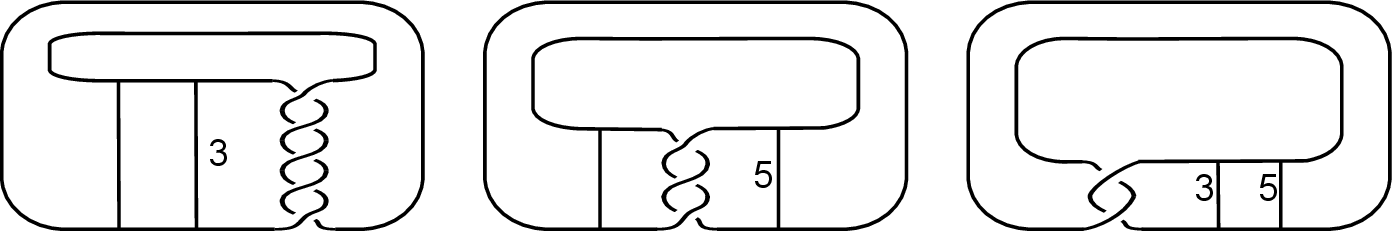}}

Figure  11
\end{center}

When $(n_1, n_2, n_3)=(2, 2, n)$, since $m_1$ and $m_2$ are
symmetric, by $(*), (**)$ we can assume $|m_1|=1, m_2=0$. Then
$n_1'=2, n_2'=1$, $m_1'=m_1, m_2'=0$, $d_1=1, d_2=2$. (1) becomes to
$$2kn_3'+m_1n_3'+2m_3'=1. \qquad (2)$$

If $k=0$, by $(*)$ we have $m_1=1$, $n_3'=1-2m_3'$. Let $m=|m_3'|$,
$d=d_3$. We have $(k, m_1, m_2, m_3, n)=(0, 1, 0, -md, (1+2m)d)$,
hence also $(0, -1, 0, md, (1+2m)d)$.

If $k=1$, by $(*)$ we have $m_1=-1$, $n_3'=1-2m_3'$. Let $m=|m_3'|$,
$d=d_3$. We have $(k, m_1, m_2, m_3, n)=(1, -1, 0, -md, (1+2m)d)$,
hence also $(-1, 1, 0, md, (1+2m)d)$.

Hence by symmetry and sign changing we have all the solutions:

\noindent $(k, m_1, m_2, m_3, n)=(0, \pm1, 0, \mp md, (1+2m)d)$,
$(\pm1, \mp1, 0, \mp md, (1+2m)d)$, $(0, 0, \pm1, \mp md, (1+2m)d)$,
or $(\pm1, 0, \mp1, \mp md, (1+2m)d)$, $d>2$. We present $(0, 1, 0,
-md, (1+2m)d)$, $m\neq0$, and $(0, 1, 0, 0, n)$, $(n>2)$.

\begin{center}
\scalebox{0.5}{\includegraphics*[0pt,0pt][512pt,155pt]{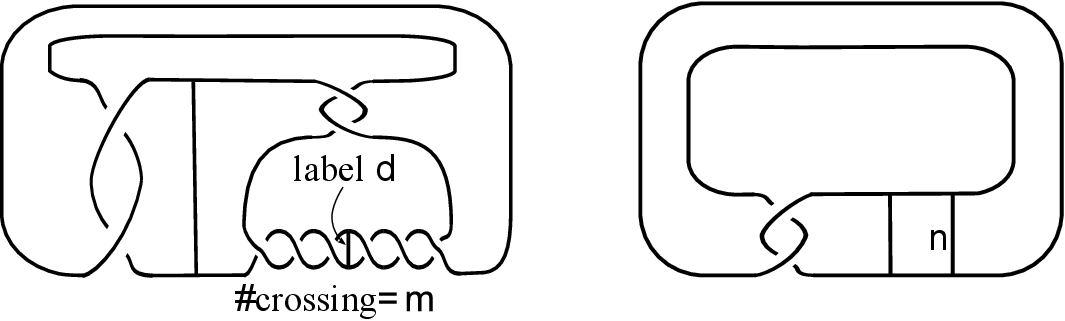}}

Figure 12
\end{center}

Denote by $\mathcal{O}$ the orbifold  presented by the left graph in
Figure 12. If we kill the edge which is labeled by $d$, we obtain an
orbifold which is fibred over $D^2(2, 2, 1+2m)$ with $1+2m\ge3$
\cite{Du1} (the 2-orbifold with base $D^2$ and three corner points
with labels $(2, 2, 1+2m)$. Then after killing the element
presenting the fiber, we obtain a surjection
$\pi_1(\mathcal{O})\rightarrow \pi_1(D^2(2, 2, 1+2m))$. The latter
group is non-abelian. But the fundament group of our 2-suborbifold,
if it exists, would be $\Z_2\times\Z_2$ after killing; this means
that we cannot find an allowable 2-suborbifold in $\mathcal{O}$.

When $(n_1, n_2, n_3)=(n, n, 1)$, $n_3'=n_3=1$, $m_3'=m_3=0$,
$d_3=1$. And (1) becomes
$$kn_1'n_2'+n_1'm_2'+n_2'm_1'=1\qquad (3).$$

By $(**)$, $\{d_1, d_2\}=\{2, d\}, (d>2)$, $\{3, 4\}$ or $\{3, 5\}$.
By symmetry we can assume $d_1<d_2$.

If $k=0$, we have
$$nm_2+nm_1=d_1d_2 \qquad(4).$$

When $d_1=2$, $d_2=d>2$, by $(4)$ we have $n=d$ or $n=2d$.

If $n=d$, then $n_1'=d/2$, $n_2'=1$, hence by $(*)$ and $(3)$ we
have $m_1'=1$, $m_2'=0$. Let $n'=d/2>1$. We have $(k, m_1, m_2,
n)=(0, 2, 0, 2n'), n'>1$.

If $n=2d$, then $n_1'=d$, $n_2'=2$, hence by $(*)$ and $(3)$ we have
$m_2'=1$, $1-2m_1'=d>2$. Let $m=|m_1'|$, we have $(k, m_1, m_2,
n)=(0, -2m, 1+2m, 2(1+2m)), m>0$.

When $d_1=3$, $d_2=4$, by $(4)$ we have $n=12$. Then $n_1'=4$,
$n_2'=3$, hence by $(*)$ and $(3)$ we have $m_1'=-1$, $m_2'=1$. We
have $(k, m_1, m_2, n)=(0, -3, 4, 12)$.

When $d_1=3$, $d_2=5$, by $(4)$ we have $n=15$. Then $n_1'=5$,
$n_2'=3$, hence by $(*)$ and $(3)$ we have $m_1'=2$, $m_2'=-1$. We
have $(k, m_1, m_2, n)=(0, 6, -5, 15)$.

If $k=1$, we have
$$n^2+nm_2+nm_1=d_1d_2 \qquad(5).$$

When $d_1=2$, $d_2=d>2$, by $(5)$ we have $n=d$ or $n=2d$.

If $n=d$, then $n_1'=d/2$, $n_2'=1$, hence by $(*)$ and $(3)$ we
have $d/2=2$, $m_1'=-1$, $m_2'=0$. We have $(k, m_1, m_2, n)=(1, -2,
0, 4)$.

If $n=2d$, then $n_1'=d$, $n_2'=2$, hence by $(*)$ and $(3)$ we have
$m_2'=-1$, $1-2m_1'=d>2$. Let $m=|m_1'|$, we have $(k, m_1, m_2,
n)=(1, -2m, -1-2m, 2(1+2m)), m>0$.

When $d_1=3$, $d_2=4$, by $(5)$ we have $n=12$. Then $n_1'=4$,
$n_2'=3$, but this contradict to $(*)$ and $(3)$.

When $d_1=3$, $d_2=5$, by $(5)$ we have $n=15$. Then $n_1'=5$,
$n_2'=3$, also contradict to $(*)$ and $(3)$.

Now by symmetry and signs changing we list all the solutions when
$(n_1, n_2, n_3)=(n, n, 1)$.

First the solutions we get above adding sign changing solutions:

$(k, m_1, m_2, n)=(0, \pm2, 0, 2n')$, $(0, \pm2m, \mp(1+2m),
2(1+2m))$, $(0, \pm3, \mp4, 12)$, $(0, \pm6, \mp5, 15)$, $(\pm1,
\mp2, 0, 4)$, $(\pm1, \mp2m, \mp(1+2m), 2(1+2m))$.

Then the symmetry solutions:

$(k, m_1, m_2, n)=(0, 0, \pm2, 2n')$, $(0, \mp(1+2m), \pm2m,
2(1+2m))$, $(0, \pm4, \mp3, 12)$, $(0, \pm5, \mp6, 15)$, $(\pm1, 0,
\mp2, 4)$, $(\pm1, \mp(1+2m), \mp2m, 2(1+2m))$.

Here the parameters satisfies $n'>1, m>0$. Symmetry solutions give
us the same picture, signs changing solutions give us the mirror
picture, and a solution with $k=1$ always has the same picture of a
solution with $k=0$. We present $(0, -1-2m, 2m, 2+4m)(m>0)$, $(0, 0,
2, 2n')(n'>1)$, $(0, 4, -3, 12)$ and $(0, 5, -6, 15)$. Note that
each of the corresponding graphs is isotopic to a simple one as
indicated in Figure 13.

\begin{center}
\scalebox{0.4}{\includegraphics*[0pt,0pt][900pt,155pt]{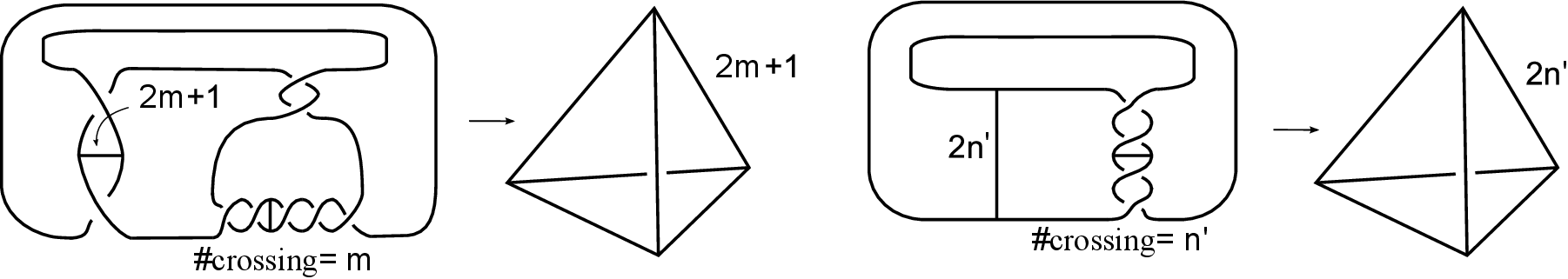}}

\scalebox{0.4}{\includegraphics*[0pt,0pt][900pt,145pt]{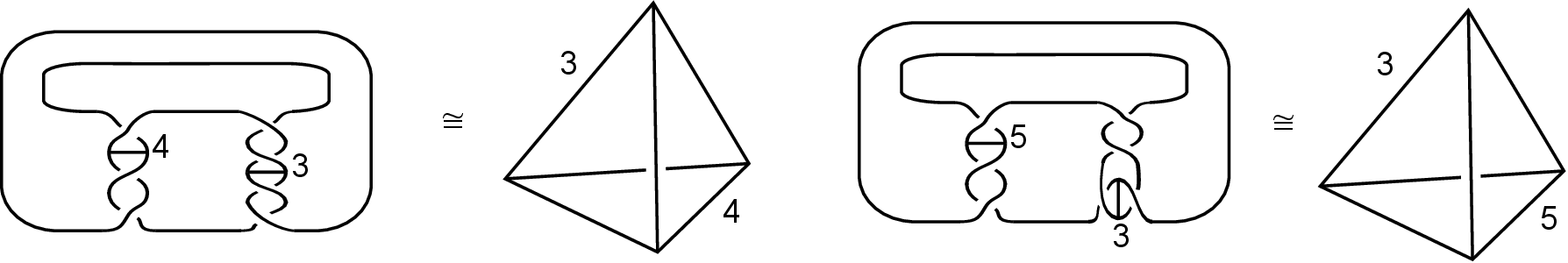}}

Figure 13
\end{center}

This finishes the discussion for Case 1.

For Case 2, $\mathcal{F}=\partial V$ has singular type $(2,2,3,3)$.
Hence $\Gamma\cap V$ must be a label $2$ arc adding two `strut
segments' with label $3$, or two arc with labels $2$ and $3$, see
Figure 14.

\begin{center}
\scalebox{0.6}{\includegraphics*[0pt,0pt][385pt,82pt]{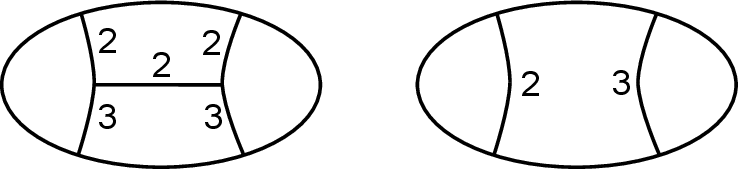}}

Figure 14
\end{center}

Then the orbifolds we find have the property that if we kill the
label $3$ `struts' or a label $3$ singular edge, we get a trivial
knot labeled by $2$.

By this property, the link cases (no vertices or struts), including
Table I now and also the graphs 14 and 16 in Table II, are easy to
handle. Graphs 04, 08, 10, 12, 13 and 14 are ruled out since each of
them has only one index. Graphs 03 and 09 are ruled out since after
killing an index 3 component (intersecting $V$), the remaining is
not a trivial knot.

For 01 we must have the two singular circle have index $2$ and $3$.
For 05 index $a$ must be $2$ and for 06 index $a$ must be $3$. For
02 index $f$ must be $3$ and $k$ is $\pm 1$, corresponds to the same
orbifold or a mirror image of 05. The orbifold presented by 16 and
02 are the same.

In Table II, there are two further graphs which possibly contain
exactly one `strut', the graphs 17 and 18. Now 17 is ruled out since
after killing the possible index $3$ `strut' (intersecting $V$), the
remaining is not a trivial knot. Concerning 18, up to mirror image,
the only possible graph is the link on the upper right hand side of
Figure 15, which presents all possible labeled links; the graph on
the lower left hand side comes from 02, 05 and 16, and the remaining
four graphs come from 01, 06, 07 and 11.

\begin{center}
\scalebox{0.5}{\includegraphics*[0pt,0pt][385pt,107pt]{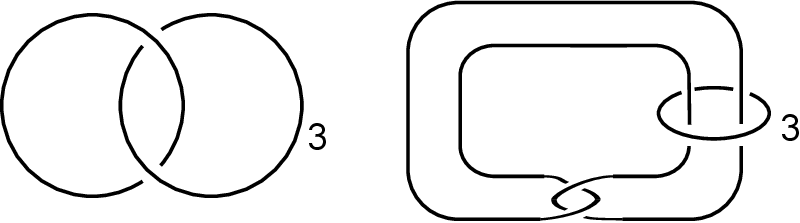}}

\scalebox{0.5}{\includegraphics*[0pt,0pt][700pt,150pt]{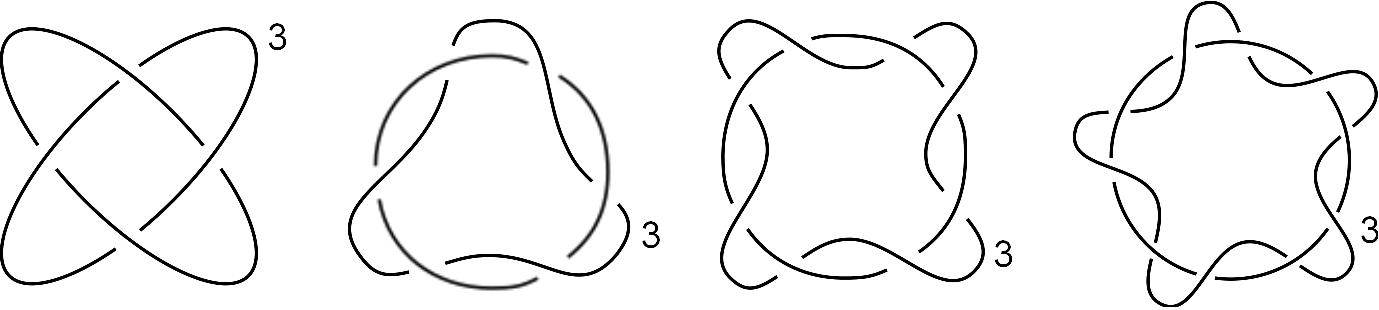}}

Figure 15
\end{center}

Concerning the possible `strut' cases, we still have to consider the
five graphs discussed in Case 1, but the case here is much simpler
since the possible `strut' can only label $3$. And after kill one or
two label $3$ `struts', we get a trivial knot labeled by $2$. Then
the case $(n_1, n_2, n_3)=(2, 3, 4)$ or $(2, 2, n)$ can be ruled
out, since after killing the index $3$ `struts', the remaining is
not a trivial knot.

We list the solutions and pictures below as in Case 1. Since most of
the solutions present the same graph or a mirror image we only
picture the graphs of non-homeomorphic orbifolds.

\noindent If $(n_1, n_2, n_3)=(2, 3, 3)$, then $(k, m_1, m_2,
m_3)=(0, \pm1, 0, 0)$, $(\pm1, \mp1, 0, 0)$, $(\pm1, \mp1, 0,
\mp1)$, $(0, \pm1, 0, \mp1)$, $(\pm1, \mp1, \mp1, 0)$ or $(0, \pm1,
\mp1, 0)$. We present $(0, 1, 0, 0)$ and $(0, -1, 0, 1)$.

\begin{center}
\scalebox{0.5}{\includegraphics*[0pt,0pt][710pt,114pt]{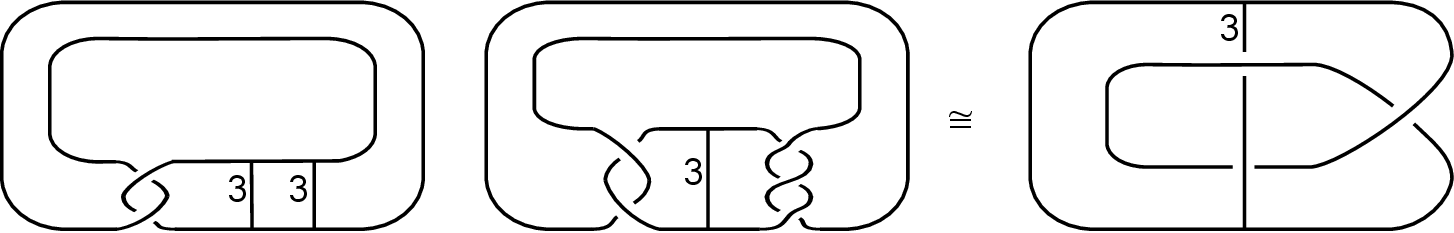}}

Figure 16
\end{center}

\noindent If $(n_1, n_2, n_3)=(2, 3, 5)$, $(k, m_1, m_2, m_3)=(\pm1,
\mp1, 0, \mp2)$ or $(0, \pm1, 0, \mp2)$. We present $(0, -1, 0, 2)$.

\begin{center}
\scalebox{0.5}{\includegraphics*[0pt,0pt][575pt,120pt]{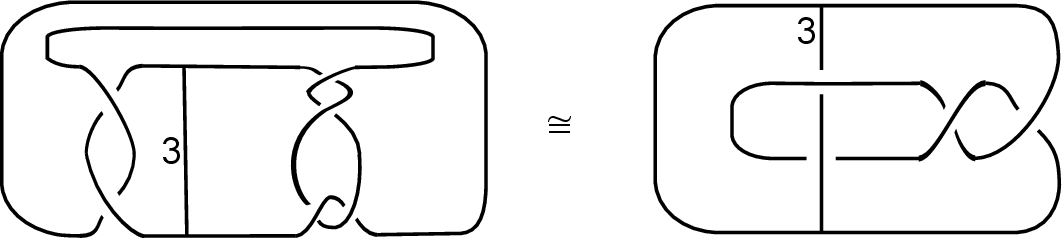}}

Figure 17
\end{center}

When $(n_1, n_2, n_3)=(n, n, 1)$, we need to solve equation $(3)$
under the following conditions:

$|2m_i|\le n_i$, $\{d_1, d_2\}=\{1, 3\}$ or $\{3, 3\}$. $(***)$

If $d_1=d_2=3$, then $n_1'=n_2'$. By $(3)$, we have $n_1'=n_2'=1$.
Hence $m_1'=m_2'=0$, $n=3$, and $k=1$. We have $(k, m_1, m_2, n)=(1,
0, 0, 3)$.

If $d_1=1$, $d_2=3$, then by $(4)$ and $(5)$ we have $n=3$. Hence
$n_1'=3$, $n_2'=1$, $m_2'=0$. Since $|2m_1'|\le n_1'$, by $(3)$,
$k=0$, $m_1'=1$. And we have $(k, m_1, m_2, n)=(0, 1, 0, 3)$.

All the solutions are $(k, m_1, m_2, n)=(\pm1, 0, 0, 3)$, $(0, \pm1,
0, 3)$ or $(0, 0, \pm1, 3)$. We present $(1, 0, 0, 3)$ and $(0, 0,
1, 3)$.

\begin{center}
\scalebox{0.5}{\includegraphics*[0pt,0pt][630pt,114pt]{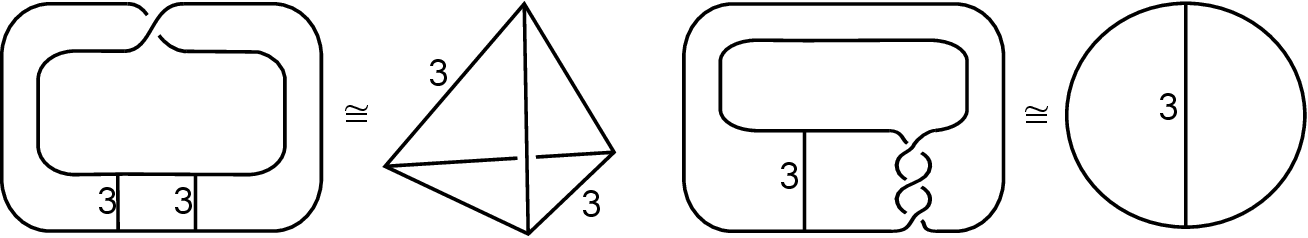}}

Figure 18
\end{center}

This finishes also the possible `strut' cases. Concluding, we have
found all fibered spherical 3-orbifolds in which an allowable
2-suborbifold might exist.

\section{List of allowable 2-suborbifolds}\label{2-orbifolds}
In this section our main result Theorem \ref{classify} will be
presented and proved. We are going to give some explanations and
conventions before we state Theorem \ref{classify}.

From now on, an edge always means an edge of $\Gamma$, the singular
set of the orbifold;  and a dashed arc is always a regular arc with
two ends at two edges of indices 2 and 3.

The primary part of Theorem \ref{classify} is the list of spherical
3-orbifolds which have survived after the discussion in Section
\ref{list 3-orbifolds}. This contains all the possible fibered
orbifolds in Table I and Table II, and all the non-fibered orbifolds
in Table III. We give each orbifold a label to indicate where it
comes from. For example, the first 3-orbifold in Table V is 15C (or
$\mathcal{O}_{15C}$), which comes from the orbifold 15 of Table II
in Dunbar's list.

We use edges and dashed arcs marked by letters $a, b, c\ldots$ to
denote 2-suborbifolds $\mathcal{F}_a, \mathcal{F}_b,
\mathcal{F}_c\ldots$ which are the boundaries of the regular
neighborhoods of these edges and arcs. We say two edges/dashed arcs
are equivalent if they can be mapped to each other by an index
preserving automorphism of $(S^3, \Gamma)$, or the boundaries of
their regular neighborhoods are orbifold isotopic to each other. That
means up to orbifold automorphism, they present the same
2-suborbifold. For an allowable 2-suborbifold, we will only mark one
edge/dashed arc in the equivalent class.

For each 3-orbifold in the list, we first give the order of its
fundamental group. Then we list allowable edges/dashed arcs and
write down the singular type of these allowable 2-suborbifold,
followed by the corresponding genus which can be computed by Lemma
\ref{compute genus}. When the singular type is $(2, 2, 3, 3)$, there
are two types denoted by I and II, corresponding to Figure 2(a) and
Figure 2(b), resp.

If the 2-suborbifold is a knotted one we give a foot notation `$k$'
to this edge or dashed arc. In the type II case, if a dashed arc can
be chosen as an unknotted one, then it can also be chosen as a
knotted one, and we add a foot notation `$uk$' to this arc.

We first list the fibred orbifolds which can contain allowable
2-suborbifold of type $(2, 2, 3, 3)$, then the fibred orbifolds that
can contain allowable 2-suborbifold of the other types, depending on
the discussion of section \ref{list 3-orbifolds}. Notice that these
two Tables have no intersection. Finally we list the non-fibred
orbifolds.

\begin{theorem}\label{classify}
Up to equivalence, the following table lists all allowable singular
edges/dashed arcs except those of type II. In the type II case, if
there exists an allowable dashed arc we give just one such dashed
arc, and this will be unknotted if there exists an unknotted one.
\end{theorem}

\begin{center}
\vskip 0.1 true cm

Table IV: Fibred case: type is $(2, 2, 3, 3)$

\vskip 0.3 true cm

\scalebox{0.4}{\includegraphics*[0pt,0pt][152pt,152pt]{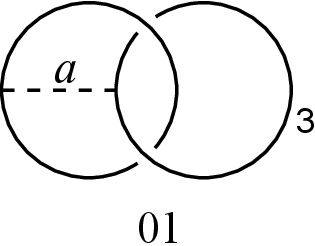}}
\raisebox{45pt} {\parbox[t]{102pt}{$|G|=6$\\$a_{uk}$: II, $g=2$}}%01
\scalebox{0.4}{\includegraphics*[0pt,0pt][152pt,152pt]{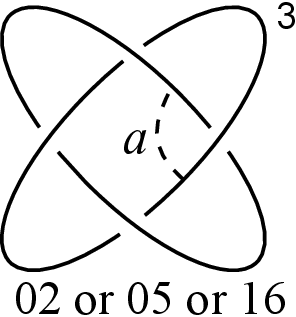}}
\raisebox{45pt} {\parbox[t]{102pt}{$|G|=18$\\$a_{uk}$: II, $g=4$}}%020516
\end{center}

\begin{center}
\scalebox{0.4}{\includegraphics*[0pt,0pt][152pt,152pt]{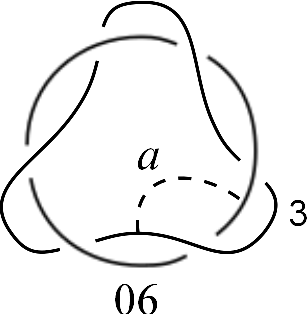}}
\raisebox{45pt} {\parbox[t]{102pt}{$|G|=48$\\$a_{uk}$: II, $g=9$}}%06
\scalebox{0.4}{\includegraphics*[0pt,0pt][152pt,152pt]{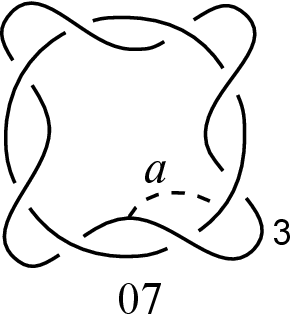}}
\raisebox{45pt} {\parbox[t]{102pt}{$|G|=144$\\$a_{uk}$: II, $g=25$}}%07
\end{center}

\begin{center}
\scalebox{0.4}{\includegraphics*[0pt,0pt][152pt,152pt]{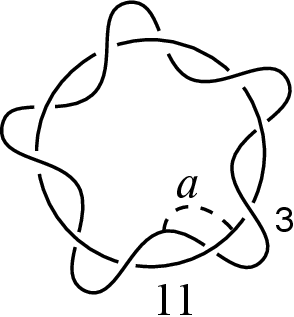}}
\raisebox{45pt} {\parbox[t]{102pt}{$|G|=720$\\$a_{uk}$: II,
$g=121$}}%11
\scalebox{0.4}{\includegraphics*[0pt,0pt][152pt,152pt]{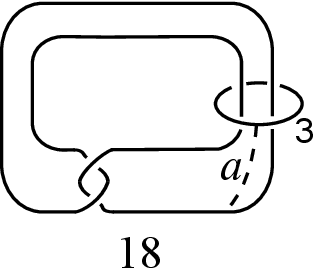}}
\raisebox{45pt} {\parbox[t]{102pt}{$|G|=144$\\$a_{uk}$: II, $g=25$}}
\end{center}

\begin{center}
\scalebox{0.4}{\includegraphics*[0pt,0pt][152pt,152pt]{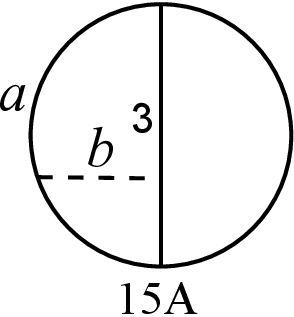}}
\raisebox{45pt} {\parbox[t]{102pt}{$|G|=6$\\$a$: I, $g=2$\\$b_{uk}$:
II, $g=2$}}%15A
\scalebox{0.4}{\includegraphics*[0pt,0pt][152pt,152pt]{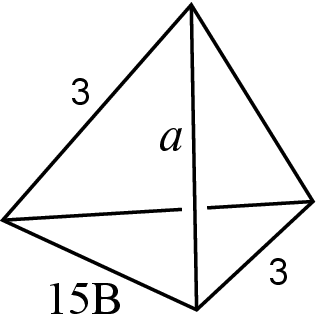}}
\raisebox{45pt} {\parbox[t]{102pt}{$|G|=18$\\$a$: I, $g=4$}}
\end{center}

\begin{center}
\scalebox{0.4}{\includegraphics*[0pt,0pt][152pt,152pt]{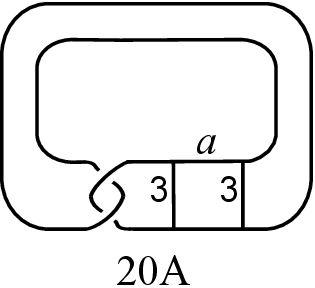}}
\raisebox{45pt} {\parbox[t]{102pt}{$|G|=144$\\$a$: I, $g=25$}}
\scalebox{0.4}{\includegraphics*[0pt,0pt][152pt,152pt]{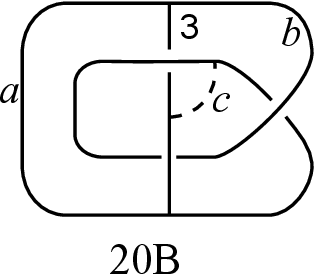}}
\raisebox{45pt} {\parbox[t]{102pt}{$|G|=48$\\$a$: I, $g=9$\\$b_k$:
I, $g=9$\\$c_{uk}$: II, $g=9$}}
\end{center}

\begin{center}
\scalebox{0.4}{\includegraphics*[0pt,0pt][152pt,152pt]{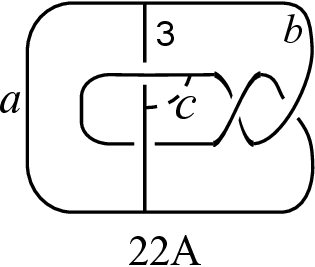}}
\raisebox{45pt} {\parbox[t]{102pt}{$|G|=720$\\$a$: I,
$g=121$\\$b_k$: I, $g=121$\\$c_{uk}$: II, $g=121$}} \hspace*{170pt}
\end{center}

\begin{center}
Table V: Fibred case: type is not $(2, 2, 3, 3)$

\vskip 0.3 true cm

\scalebox{0.4}{\includegraphics*[0pt,0pt][152pt,152pt]{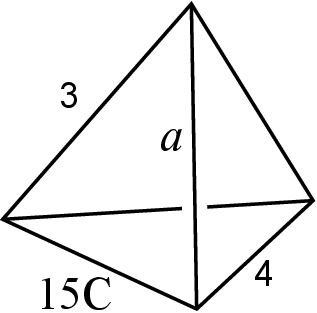}}
\raisebox{45pt} {\parbox[t]{102pt}{$|G|=24$\\$a$:(2,2,3,4), $g=6$}}
\scalebox{0.4}{\includegraphics*[0pt,0pt][152pt,152pt]{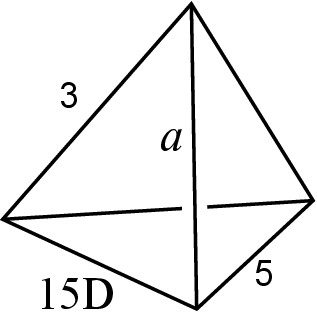}}
\raisebox{45pt} {\parbox[t]{102pt}{$|G|=30$\\$a$:(2,2,3,5), $g=8$}}
\end{center}

\begin{center}
\scalebox{0.4}{\includegraphics*[0pt,0pt][152pt,152pt]{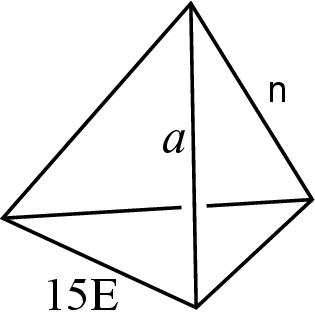}}
\raisebox{45pt}{\parbox[t]{102pt}{$|G|=4n$\\$a$:(2,2,2,n)\\$g=n-1$}}
\scalebox{0.4}{\includegraphics*[0pt,0pt][152pt,152pt]{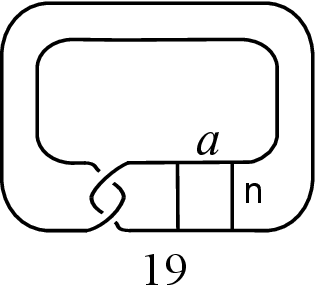}}
\raisebox{45pt}{\parbox[t]{102pt}{$|G|=4n^2$\\$a$:(2,2,2,n)\\$g=(n-1)^2$}}
\end{center}

\begin{center}
\scalebox{0.4}{\includegraphics*[0pt,0pt][152pt,152pt]{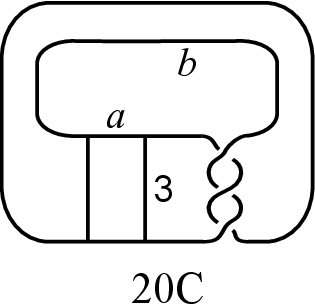}}
\raisebox{45pt} {\parbox[t]{102pt}{$|G|=96$\\$a$:(2,2,2,3),
$g=9$\\$b_{k}$:(2,2,2,3), $g=9$}}
\scalebox{0.4}{\includegraphics*[0pt,0pt][152pt,152pt]{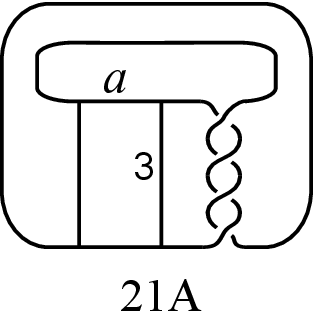}}
\raisebox{45pt} {\parbox[t]{102pt}{$|G|=288$\\$a$:(2,2,2,3),
$g=25$}}
\end{center}

\begin{center}
\scalebox{0.4}{\includegraphics*[0pt,0pt][152pt,152pt]{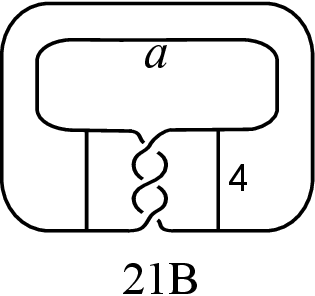}}
\raisebox{45pt} {\parbox[t]{102pt}{$|G|=384$\\$a$:(2,2,2,4),
$g=49$}}
\scalebox{0.4}{\includegraphics*[0pt,0pt][152pt,152pt]{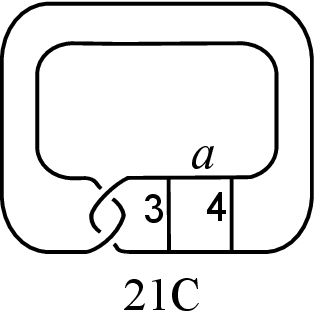}}
\raisebox{45pt} {\parbox[t]{102pt}{$|G|=576$\\$a$:(2,2,3,4),
$g=121$}}
\end{center}

\begin{center}
\scalebox{0.4}{\includegraphics*[0pt,0pt][152pt,152pt]{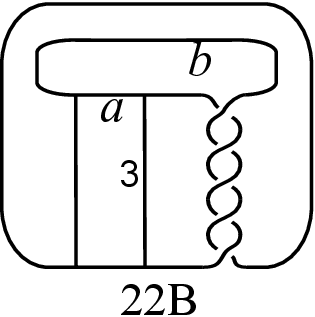}}
\raisebox{45pt} {\parbox[t]{102pt}{$|G|=1440$\\$a$:(2,2,2,3),
$g=121$\\$b_{k}$:(2,2,2,3), $g=121$}}
\scalebox{0.4}{\includegraphics*[0pt,0pt][152pt,152pt]{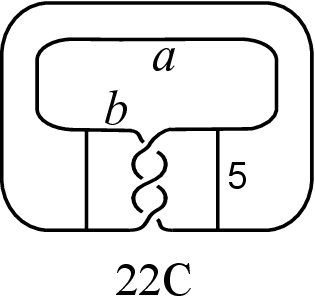}}
\raisebox{45pt} {\parbox[t]{102pt}{$|G|=2400$\\$a$:(2,2,2,5),
$g=361$\\$b_{k}$:(2,2,2,5), $g=361$}}
\end{center}

\begin{center}
\scalebox{0.4}{\includegraphics*[0pt,0pt][152pt,152pt]{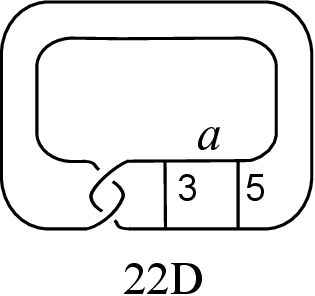}}
\raisebox{45pt} {\parbox[t]{102pt}{$|G|=3600$\\$a$:(2,2,3,5),
$g=841$}}\hspace*{170pt}
\end{center}

\begin{center}
Table VI: Non-fibred case

\vskip 0.3 true cm

\scalebox{0.4}{\includegraphics*[0pt,0pt][152pt,152pt]{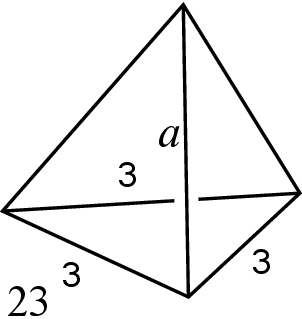}}
\raisebox{45pt} {\parbox[t]{102pt}{$|G|=96$\\$a$: I, $g=17$}}
\scalebox{0.4}{\includegraphics*[0pt,0pt][152pt,152pt]{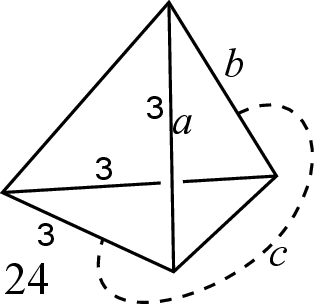}}
\raisebox{45pt} {\parbox[t]{102pt}{$|G|=60$\\$a$:(2,2,2,3),
$g=6$\\$b$: I, $g=11$\\$c_{k}$: II, $g=11$}}
\end{center}

\begin{center}
\scalebox{0.4}{\includegraphics*[0pt,0pt][152pt,152pt]{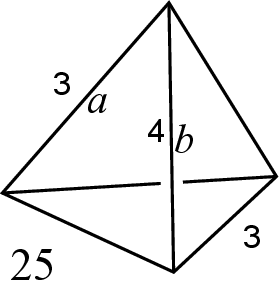}}
\raisebox{45pt} {\parbox[t]{102pt}{$|G|=576$\\$a$:(2,2,2,4),
$g=73$\\$b$: I, $g=97$}}
\scalebox{0.4}{\includegraphics*[0pt,0pt][152pt,152pt]{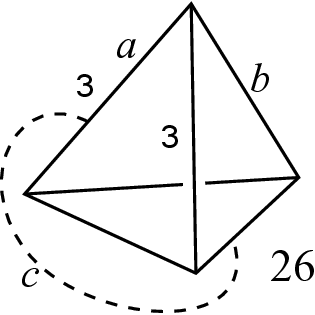}}
\raisebox{45pt} {\parbox[t]{102pt}{$|G|=24$\\$a$:(2,2,2,3),
$g=3$\\$b$: I, $g=5$\\$c_k$: II, $g=5$}}
\end{center}

\begin{center}
\scalebox{0.4}{\includegraphics*[0pt,0pt][152pt,152pt]{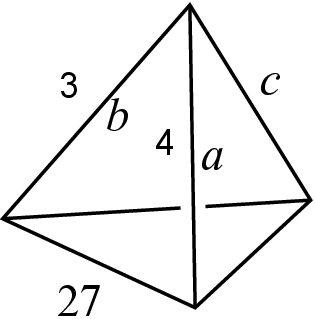}}
\raisebox{45pt} {\parbox[t]{102pt}{$|G|=48$\\$a$:(2,2,2,3),
$g=5$\\$b$:(2,2,2,4), $g=7$\\$c$:(2,2,3,4), $g=11$}}
\scalebox{0.4}{\includegraphics*[0pt,0pt][152pt,152pt]{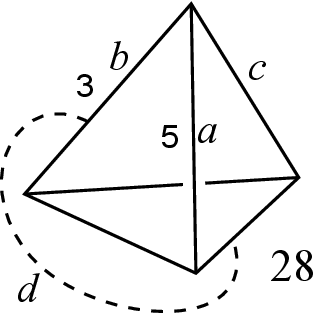}}
\raisebox{45pt} {\parbox[t]{102pt}{$|G|=120$\\$a$:(2,2,2,3),
$g=11$\\$b$:(2,2,2,5), $g=19$\\$c$:(2,2,3,5), $g=29$\\$d_{k}$: II,
$g=21$}}
\end{center}

\begin{center}
\scalebox{0.4}{\includegraphics*[0pt,0pt][152pt,152pt]{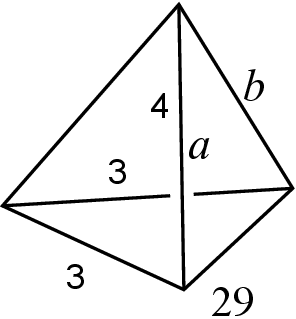}}
\raisebox{45pt} {\parbox[t]{102pt}{$|G|=192$\\$a$:(2,2,2,3),
$g=17$\\$b$:(2,2,3,4), $g=41$}}
\scalebox{0.4}{\includegraphics*[0pt,0pt][152pt,152pt]{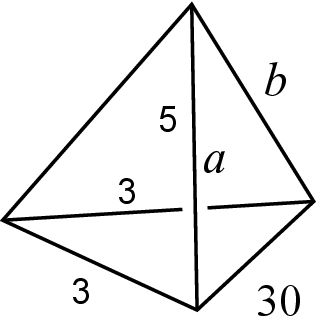}}
\raisebox{45pt} {\parbox[t]{102pt}{$|G|=7200$\\$a$:(2,2,2,3),
$g=601$\\$b$:(2,2,3,5), $g=1681$}}
\end{center}

\begin{center}
\scalebox{0.4}{\includegraphics*[0pt,0pt][152pt,152pt]{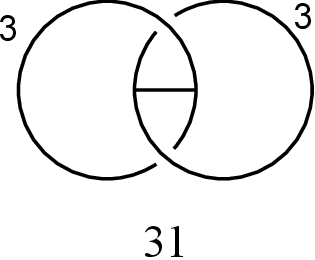}}
\raisebox{45pt} {\parbox[t]{102pt}{$|G|=288$\\No allowable\\
2-suborbifold}}
\scalebox{0.4}{\includegraphics*[0pt,0pt][152pt,152pt]{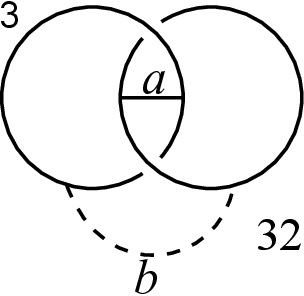}}
\raisebox{45pt} {\parbox[t]{102pt}{$|G|=24$\\$a$: I,
$g=5$\\$b_{uk}$: II, $g=5$}}
\end{center}

\begin{center}
\scalebox{0.4}{\includegraphics*[0pt,0pt][152pt,152pt]{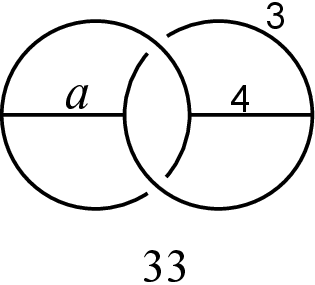}}
\raisebox{45pt} {\parbox[t]{102pt}{$|G|=1152$\\$a$:(2,2,2,3),
$g=97$}}
\scalebox{0.4}{\includegraphics*[0pt,0pt][152pt,152pt]{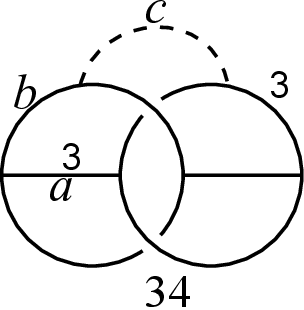}}
\raisebox{45pt} {\parbox[t]{102pt}{$|G|=120$\\$a$:(2,2,2,3),
$g=11$\\$b_{k}$:(2,2,2,3), $g=11$\\$c_{k}$: II, $g=21$}}
\end{center}

\begin{center}
\scalebox{0.4}{\includegraphics*[0pt,0pt][152pt,152pt]{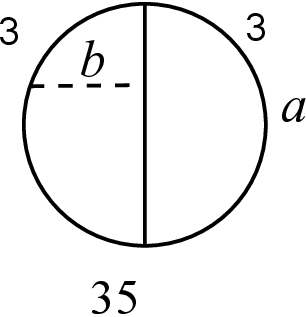}}
\raisebox{45pt} {\parbox[t]{102pt}{$|G|=12$\\$a$: I,
$g=3$\\$b_{uk}$: II, $g=3$}}
\scalebox{0.4}{\includegraphics*[0pt,0pt][152pt,152pt]{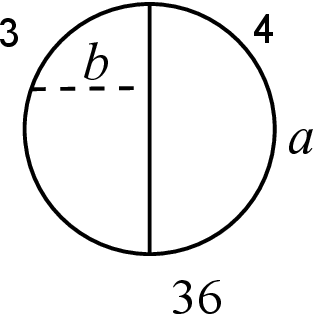}}
\raisebox{45pt} {\parbox[t]{102pt}{$|G|=24$\\$a$: I,
$g=5$\\$b_{uk}$: II, $g=5$}}
\end{center}

\begin{center}
\scalebox{0.4}{\includegraphics*[0pt,0pt][152pt,152pt]{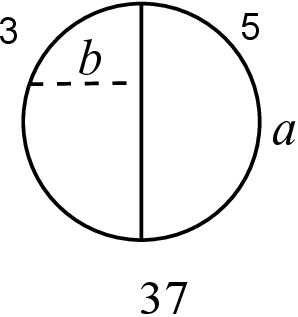}}
\raisebox{45pt} {\parbox[t]{102pt}{$|G|=60$\\$a$: I,
$g=11$\\$b_{uk}$: II, $g=11$}}
\scalebox{0.4}{\includegraphics*[0pt,0pt][152pt,152pt]{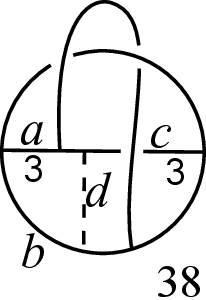}}
\raisebox{45pt} {\parbox[t]{102pt}{$|G|=2880$\\$a, b_k,
c_k$:(2,2,2,3)\\$g=241$\\$d_k$: II, $g=481$}}
\end{center}

\begin{center}
\scalebox{0.4}{\includegraphics*[0pt,0pt][152pt,152pt]{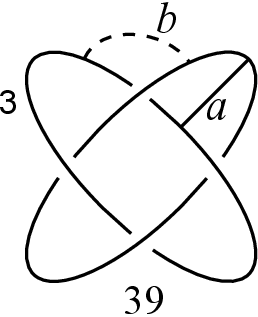}}
\raisebox{45pt} {\parbox[t]{102pt}{$|G|=576$\\$a$: I,
$g=97$\\$b_{uk}$: II, $g=97$}}
\scalebox{0.4}{\includegraphics*[0pt,0pt][152pt,152pt]{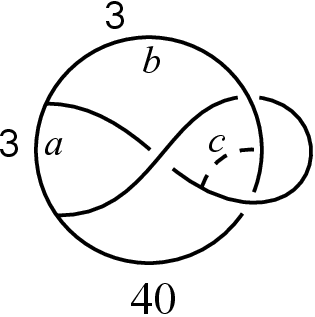}}
\raisebox{45pt} {\parbox[t]{102pt}{$|G|=1440$\\$a$: I,
$g=241$\\$b_k$: I, $g=241$\\$c_{uk}$: II, $g=241$}}
\end{center}

\begin{remark}
From the list it is easy to see when $g = 21$ or $g = 481$ the
maximal order can only be realized by a knotted embedding. The
orbifolds corresponding to these situations are $\mathcal{O}_{28}$,
$\mathcal{O}_{34}$ and $\mathcal{O}_{38}$.
\end{remark}

\begin{proof}
One can easily check that the list of 3-orbifolds in Theorem
\ref{classify} contains exactly those in Tables I and II which
survive after Section \ref{list 3-orbifolds}, and those in Table
III. By the discussion of Section \ref{list 3-orbifolds}, all
allowable 2-orbifolds are contained in one of these 3-orbifolds.

Then we need compute the orders of fundamental groups of these
3-orbifolds. There are two infinity sequences of 3-orbifolds, 15E
and 19 in the above list. For these two sequences, we know the group
actions on the pair $(S^3,\Sigma_g)$ clearly, see Example \ref{cage}
below, also see \cite{WWZZ}. The other finitely many orders of the
fundamental groups can be calculated directly from their Wirtinger
presentations of the graphs (or see proposition 5 in \cite{Du1} for
a presentation) and by using \cite{GAP}. For the non-fibred case,
the group orders can also be got from the group structure, see
\cite{Du2}.

For the details about how to use the computer software \cite{GAP}, one can see Section 8.

\begin{example}\label{O34}
We give here an example to explain how we get the group order. This
is the orbifold $\mathcal{O}_{34}$ in the above list. It is a
non-fibred spherical orbifold.

\begin{center}
\scalebox{0.6}{\includegraphics{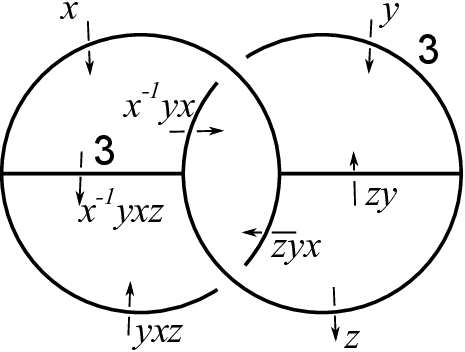}}

Figure 19
\end{center}

From Figure 19 we obtain the following presentation of the orbifold
fundamental group of $\mathcal{O}_{34}$:
$$\pi_1(\mathcal{O}_{34})=\langle x, y, z \mid x^2,y^3,z^2,(zy)^2,(yxz)^2,(yxzx)^3\rangle$$
We input these generators and relations into \cite{GAP}, and the
computer uses the standard procedure (called coset enumeration) to
show the group order is $120$.

There is also another way to get this. In \cite{Du2} we know
$\pi_1(\mathcal{O}_{34})\cong
\mathbf{J}\times^\mathbf{*}_{\mathbf{J}}\mathbf{J}$, which can map
surjectively to $J\times^*_J J$ under the two to one homomorphism
$SO(4)\rightarrow SO(3)\times SO(3)$. Hence
$|\pi_1(\mathcal{O}_{34})|=2\times 60 \times 60\div 60=120$.
\end{example}

Next we will check graph by graph in the above list as following:

Step 1: Up to equivalence, list all candidacy edges (see Definition
4.2).

Step 2: Check whether the edges given in Step 1 is allowable.

Unknotted edges automatically satisfies the surjection condition
(Remark \ref{surjective}). For knotted edges, we use the so-called
coset enumeration method (\cite{Ro}, p. 351, Chapter 11) to verify
the surjectivity. For a knotted allowable edge, we give it a label
`k'. After these two Steps we can list all allowable edges.

Step 3: Try to find an unknotted candidacy dashed arc.

If there is one, then the surjection condition is automatically
satisfied, and we labeled the dashed arc by `uk'. Otherwise we will
prove there is no unknotted allowable dashed arc.

Step 4: If there is no unknotted allowable dashed arc, then we will
try to find a knotted allowable dashed arc and give it a label `k'.

Step 5: If in an 3-orbifold in the list there is no marked dashed
arc, we will proof there is no allowable dashed arc in the orbifold.

After we list all the allowable edges and dashed arcs, the genus can
be computed from the group order and the singular type by Lemma
\ref{compute genus}. And the proof will be finished.

We firstly check Table IV, then Table V, and finally Table VI.

\noindent {\bf Orbifolds in Table IV:}

01, 02, 06, 07, 11, 18: In these six orbifolds there are no
candidacy edges and the possible unknotted dashed arcs have been
marked.

15A: Up to equivalence, there is only one candidacy edge $a$ which
is unknotted, hence allowable. And there is an unknotted allowable
dashed arc $b$.

15B: Up to equivalence, there is only one candidacy edge $a$ which
is unknotted, hence allowable. For a candidacy dashed arc, one end
of it must lie on an edge labeled $3$. Kill this edge, then
$\pi_1(\mathcal{F})=\Z_2$, but $\pi_1((S^3, \Gamma'))=D_3$. Hence
the dashed arc can not be allowable. There is no allowable dashed
arc.

20A: Up to equivalence, there are two candidacy edges. $a$ is
unknotted, hence allowable. The other one has its two vertices on
the same singular edge labeled $3$. Kill this edge, then
$\pi_1(\mathcal{F})=\Z_2$, but $\pi_1((S^3, \Gamma'))=D_3$. Hence it
is not allowable. Similarly, for a candidacy dashed arc, one end of
it must lie on one edge labeled $3$. Kill this edge, then
$\pi_1(\mathcal{F})=\Z_2$, but $\pi_1((S^3, \Gamma'))=D_3$. We know
there is no allowable dashed arc.

20B, 22A: There are two candidacy edges $a$ and $b$. $a$ is
unknotted. We also give an unknotted candidacy dashed arc $c$. Hence
$a$ and $c$ are allowable. We need \cite{GAP} to check that the
knotted edge $b$ is also allowable.

For 20B,
\begin{center}
\scalebox{0.8}{\includegraphics{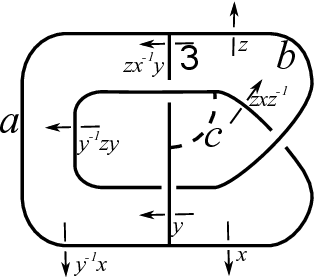}}

Figure 20
\end{center}

$$\pi_1(\mathcal{O}_{20B})=\langle x,y,z| y^{-1}zyzx^{-1}z^{-1},z^2,(y^{-1}x)^2,y^3\rangle.$$

For $b$ in 20B, the corresponding subgroup is
$$i_*(\pi_1(\mathcal{F}_b))=\langle z, y^{-1}x, yzyz^{-1}y^{-1}\rangle.$$

By \cite{GAP}, the computer can show that
$[\pi_1(\mathcal{O}_{20B}):i_*(\pi_1(\mathcal{F}_b))]=1$, so $b$ is
allowable.

For 22A,
\begin{center}
\scalebox{0.8}{\includegraphics{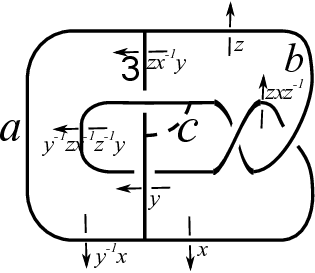}}

Figure 21
\end{center}

$$\pi_1(\mathcal{O}_{22A})=\langle x,y,z|y^{-1}zx^{-1}z^{-1}yzx^{-1}zxz^{-1},
z^2,(y^{-1}x)^2,y^3\rangle.$$

For $b$ in 22A, the corresponding subgroup is
$$i_*(\pi_1(\mathcal{F}_b))=\langle z, y^{-1}x, (zxz^{-1}y^{-1}z)y(zxz^{-1}y^{-1}z)^{-1}\rangle.$$

By \cite{GAP}, the computer can show that
$[\pi_1(\mathcal{O}_{22A}):i_*(\pi_1(\mathcal{F}_b))]=1$, so $b$ is
allowable.

\noindent {\bf Orbifolds in Table V:}

By the discussion in Section \ref{list 3-orbifolds}, orbifolds in
Table V contains no allowable edges/dashed arcs of type $(2, 2, 3,
3)$, hence contains no allowable dashed arc. We need only list all
allowable edges.

15C, 15D, 15E: Up to equivalence, there is only one candidacy edge
$a$ which is unknotted, hence allowable.

19, 21A, 21C, 22D: Up to equivalence, there are two candidacy edges.
$a$ is unknotted, hence allowable. The other one has its two
vertices on the same labeled $3$ singular edge, hence has type $(2,
2, 3, 3)$ and is not allowable by the discussion in Section
\ref{list 3-orbifolds}.

20C, 22B, 22C: Up to equivalence, there are two candidacy edges. $a$
is unknotted, hence allowable. $b$ is knotted, we need \cite{GAP} to
check it is allowable.

For 20C,
\begin{center}
\scalebox{0.8}{\includegraphics{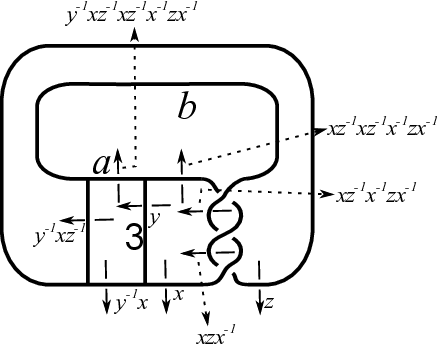}}

Figure 22
\end{center}

$$\pi_1(\mathcal{O}_{20C})=\langle x,y,z|
x^2,z^2,y^3,(y^{-1}x)^2,(y^{-1}xz^{-1})^2,(y^{-1}xz^{-1}xz^{-1}x^{-1}zx^{-1})^2\rangle.$$

For $b$ in 20C, the corresponding subgroup is
$$i_*(\pi_1(\mathcal{F}_b))=\langle xz^{-1}x^{-1}zx^{-1}, y^{-1}xz^{-1},(xz^{-1}x^{-1})y(xz^{-1}x^{-1})^{-1}\rangle.$$

By \cite{GAP}, the computer can show that
$[\pi_1(\mathcal{O}_{20C}):i_*(\pi_1(\mathcal{F}_b))]=1$, so $b$ is
allowable.

Below in Figure 23 and 24 we have used the relation $x^2=z^2=1$ to
simplify the notations.

For 22B,
\begin{center}
\scalebox{0.8}{\includegraphics{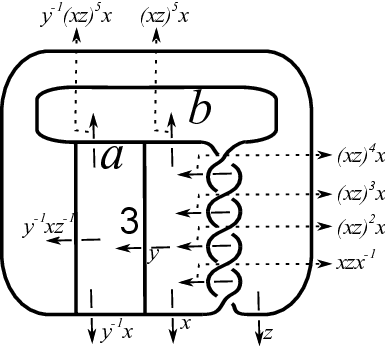}}

Figure 23
\end{center}

$$\pi_1(\mathcal{O}_{22B})=\langle x,y,z| x^2,y^3,z^2,(y^{-1}x)^2,(y^{-1}(xz)^5x)^2,(y^{-1}xz^{-1})^2\rangle.$$

For $b$ in 22B,  the corresponding subgroup is
$$i_*(\pi_1(\mathcal{F}_b))=\langle y^{-1}(xz)^5x, y^{-1}xz^{-1},(xzxz)y(xzxz)^{-1}\rangle.$$

By \cite{GAP}, the computer can show that
$[\pi_1(\mathcal{O}_{22B}):i_*(\pi_1(\mathcal{F}_b))]=1$, so $b$ is
allowable.

For 22C,
\begin{center}
\scalebox{0.8}{\includegraphics{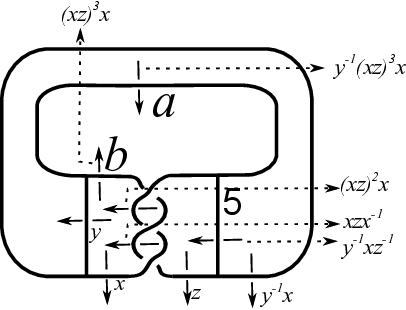}}

Figure 24
\end{center}

$$\pi_1(\mathcal{O}_{22C})=\langle x,y,z| x^2,y^2,z^2,(y^{-1}x)^2,(y^{-1}(xz)^3x)^2,(y^{-1}xz^{-1})^5\rangle.$$

For $b$ in 22C, the corresponding subgroup is
$$i_*(\pi_1(\mathcal{F}_b))=\langle y, y^{-1}(xz)^3x,(xzxz)y^{-1}x(xzxz)^{-1}\rangle.$$

By \cite{GAP}, the computer can show that
$[\pi_1(\mathcal{O}_{22C}):i_*(\pi_1(\mathcal{F}_b))]=1$, so $b$ is
allowable.

21B: There are two classes of candidacy edges, $\{a,a'\}$ and
$\{b,c\}$. $a$ is unknotted, hence allowable. To see $b$ is not
allowable, kill $a$ and $a'$, then $\pi_1(\mathcal{F})=\Z_2$, but
$\pi_1((S^3, \Gamma'))$ is $D_3$.

\begin{center}
\scalebox{0.6}{\includegraphics{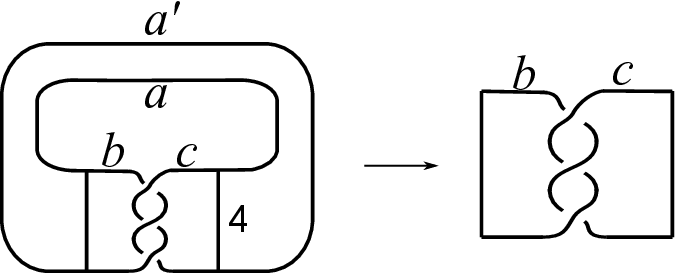}}

Figure 25
\end{center}

\noindent {\bf Orbifolds in Table VI:}

25, 27: All candidacy edges are unknotted and have been marked as in
the orbifold, up to equivalence. For a candidacy dashed arc, one end
of it must lie on some edge labeled $3$. Kill this edge, then
$\pi_1(\mathcal{F})=\Z_2$, but $\pi_1((S^3, \Gamma'))=D_3$. Hence
there is no allowable dashed arc.

24, 26, 28: All candidacy edges are unknotted and have been marked
as in the orbifold, up to equivalence. For a candidacy dashed arc
corresponding to a 2-suborbifold $\mathcal{F}$, the other side of
$\mathcal{F}$, containing four vertices, can not be a handlebody
orbifold. Hence there is no unknotted allowable dashed arc. Then we
need \cite{GAP} to check that $c_k$ in 24 and 26, $d_k$ in 28 are
knotted allowable dashed arcs.

For 24,
\begin{center}
\scalebox{0.8}{\includegraphics{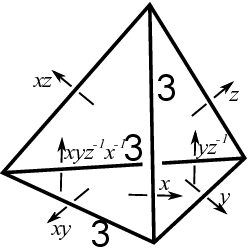}}

Figure 26
\end{center}

$$\pi_1(\mathcal{O}_{24})=\langle x,y,z| x^3,y^2,z^2,(xz)^2,(xy)^3,(yz^{-1})^3\rangle.$$

For $c$ in 24, the corresponding subgroup is
$$i_*(\pi_1(\mathcal{F}_c))=\langle xy, z\rangle.$$

By \cite{GAP}, the computer can show that
$[\pi_1(\mathcal{O}_{24}):i_*(\pi_1(\mathcal{F}_c))]=1$, so $c$ is
allowable.

For 26,
\begin{center}
\scalebox{0.8}{\includegraphics{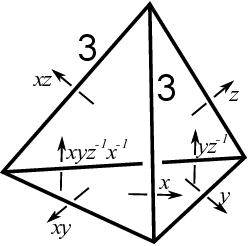}}

Figure 27
\end{center}

$$\pi_1(\mathcal{O}_{26})=\langle x,y,z| x^3,y^2,z^2,(xz)^3,(xy)^2,(yz^{-1})^2\rangle.$$

For $c$ in 26, the corresponding subgroup is
$$i_*(\pi_1(\mathcal{F}_c))=\langle xz, y\rangle.$$

By \cite{GAP}, the computer can show that
$[\pi_1(\mathcal{O}_{26}):i_*(\pi_1(\mathcal{F}_c))]=1$, so $c$ is
allowable.

For 28,
\begin{center}
\scalebox{0.8}{\includegraphics{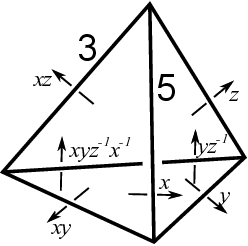}}

Figure 28
\end{center}

$$\pi_1(\mathcal{O}_{28})=\langle x,y,z| x^5,y^2,z^2,(xz)^3,(xy)^2,(yz^{-1})^2\rangle.$$

For $d$ in 28, the corresponding subgroup is
$$i*(\pi_1(\mathcal{F}_d))=\langle xz, y\rangle.$$

By \cite{GAP}, the computer can show that
$[\pi_1(\mathcal{O}_{28}):i_*(\pi_1(\mathcal{F}_d))]=1$, so $d$ is
allowable.

31: There is no candidacy edge. By killing one labeled $3$ edge and
the labeled $2$ arc, it is easy to see there is no allowable dashed
arc.

32, 35, 36, 37, 39: Up to equivalence, there is only one candidacy
edge $a$ which is unknotted, hence allowable. And there is an
unknotted allowable dashed arc $b$.

33: There are three classes of candidacy edges, $\{a,a'\}$,
$\{b,d\}$ and $e$. $a$ is unknotted, hence allowable. For $b$, we
kill $d$, then $\pi_1(\mathcal{F})=\Z_2$, but $\pi_1((S^3,
\Gamma'))$ is $D_3$. For $e$, we kill $f$, then $\pi_1(\mathcal{F})$
is trivial, but $\pi_1((S^3, \Gamma'))$ is $\Z_2$. Hence $b$ and $e$
are not allowable. For a candidacy dashed arc, one end of it must
lie on $e$. Kill $e$, then $\pi_1(\mathcal{F})=\Z_2$, but
$\pi_1((S^3, \Gamma'))=D_2$. Hence there is no allowable dashed arc.

\begin{center}
\scalebox{0.7}{\includegraphics{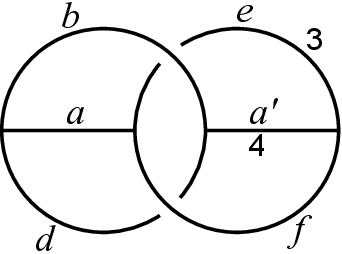}}

Figure 29
\end{center}

34: There are three classes of candidacy edges, $\{a,a'\}$,
$\{b,d\}$ and $e$. $a$ is unknotted, hence allowable. For $e$, we
kill $f$, then $\pi_1(\mathcal{F})$ is trivial, but $\pi_1((S^3,
\Gamma'))$ is $\Z_2$. Hence $e$ is not allowable. This graph
contains four vertices, hence there is no unknotted allowable dashed
arc. Then we need \cite{GAP} to check that the knotted edge $b$ and
the knotted dashed arc $c$ are allowable.

\begin{center}
\scalebox{0.5}{\includegraphics{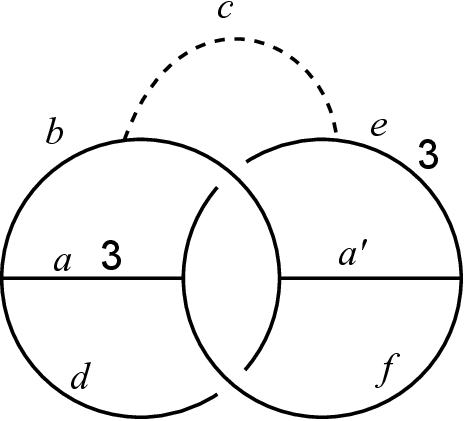}}

Figure 30
\end{center}

By Example \ref{O34}, we have known that
$$\pi_1(\mathcal{O}_{34})=\langle x, y, z \mid x^2,y^3,z^2,(zy)^2,(yxz)^2,(yxzx)^3\rangle.$$

For $b$, the corresponding subgroup is
$$i_*(\pi_1(\mathcal{F}_b))=\langle yxz, x, zy\rangle.$$

By \cite{GAP}, the computer can show that
$[\pi_1(\mathcal{O}):i_*(\pi_1(\mathcal{F}_b))]=1$, so $b$ is
allowable.

For $c$, the corresponding subgroup is
$$i_*(\pi_1(\mathcal{F}_c))=\langle x, y\rangle.$$

Also by \cite{GAP}, the computer can show that
$[\pi_1(\mathcal{O}):i_*(\pi_1(\mathcal{F}_c))]=1$, so $c$ is
allowable.

38: All the $6$ edges are candidacy. $a$ and $a'$ are equivalent.
They are unknotted, hence allowable. Since the graph contains four
vertices, there is no unknotted allowable dashed arc. Then we need
\cite{GAP} to check that $b$, $c$ and $d$ are allowable. They are
all knotted. We also need \cite{GAP} to check that $e$ and $f$ are
not allowable.

\begin{center}
\scalebox{0.8}{\includegraphics{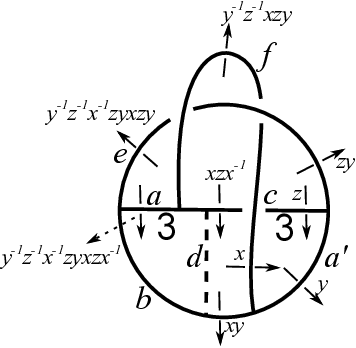}}

Figure 31
\end{center}

$$\pi_1(\mathcal{O}_{38})=\langle x,y,z| x^2,y^2,z^3,(xy)^2,(zy)^2,(y^{-1}z^{-1}x^{-1}zyxzx^{-1})^3\rangle.$$

For $b$ in 38, the corresponding subgroup is
$$i_*(\pi_1(\mathcal{F}_b))=\langle xy, x, y^{-1}z^{-1}x^{-1}zyxzy\rangle.$$

By \cite{GAP}, the computer can show that
$[\pi_1(\mathcal{O}_{38}):i_*(\pi_1(\mathcal{F}_b))]=1$, so $b$ is
allowable.

For $c$ in 38, the corresponding subgroup is
$$i_*(\pi_1(\mathcal{F}_c))=\langle xzx^{-1}, y^{-1}z^{-1}xzy,xyx^{-1}\rangle.$$

By \cite{GAP}, the computer can show that
$[\pi_1(\mathcal{O}_{38}):i_*(\pi_1(\mathcal{F}_c))]=1$, so $c$ is
allowable.

For $d$ in 38, the corresponding subgroup is
$$i_*(\pi_1(\mathcal{F}_d))=\langle xzx^{-1}, xy\rangle.$$

By \cite{GAP}, the computer can show that
$[\pi_1(\mathcal{O}_{38}):i_*(\pi_1(\mathcal{F}_d))]=1$, so $d$ is
allowable.

For $e$ in 38, the corresponding subgroup is
$$i_*(\pi_1(\mathcal{F}_e))=\langle y^{-1}z^{-1}x^{-1}zyxzy, xy, (y^{-1}z^{-1}xzy)^{-1}y(y^{-1}z^{-1}xzy)\rangle.$$

By \cite{GAP}, the computer can show that
$[\pi_1(\mathcal{O}_{38}):i_*(\pi_1(\mathcal{F}_e))]=4$, so $e$ is
not allowable.

For $f$ in 38, the corresponding subgroup is
$$i_*(\pi_1(\mathcal{F}_f))=\langle y^{-1}z^{-1}xzy, xzx^{-1},(zy)^{-1}y(zy)\rangle.$$

By \cite{GAP}, the computer can show that
$[\pi_1(\mathcal{O}_{38}):i_*(\pi_1(\mathcal{F}_f))]=5$, so $f$ is
not allowable.

40: There are only two candidacy edges $a$ and $b$. And there is a
candidacy dashed arc $c$. $a$ and $c$ are unknotted, hence both
allowable. $b$ is knotted, we need \cite{GAP} to check it is
allowable.

\begin{center}
\scalebox{0.8}{\includegraphics{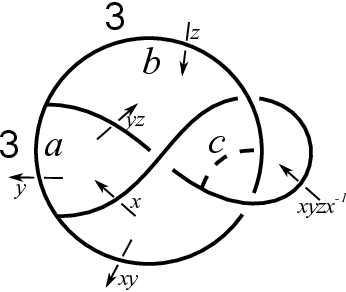}}

Figure 32
\end{center}

$$\pi_1(\mathcal{O}_{40})=\langle x,y,z| x^2,y^3,z^3,zx^{-1}z^{-1}=xyzx^{-1}\rangle.$$

For $b$ in 40, the corresponding subgroup is
$$i_*(\pi_1(\mathcal{F}_b))=\langle y, z, (xyzx^{-1})^{-1}x(xyzx^{-1})\rangle.$$

By \cite{GAP}, the computer can show that
$[\pi_1(\mathcal{O}_{40}):i_*(\pi_1(\mathcal{F}_b))]=1$, so $b$ is
allowable.

23, 29, 30: All the candidacy edges are unknotted and have been
marked as in the orbifold, up to equivalence. We use Lemma
\ref{surjection} below to show there is no allowable dashed arc.

The fundamental groups of these three orbifolds are the finite
groups $\mathbf{T}\times_{\mathbf{C}_3}\mathbf{T}$,
$\mathbf{O}\times_{\mathbf{D}_3}\mathbf{O}$ and
$\mathbf{J}\times\mathbf{J}$ of $SO(4)$, see \cite{Du2} for the
notations. They can be mapped surjectively to $T\times_{C_3}T$,
$O\times_{D_3}O$, $J\times J$ under the $2$ to $1$ map
$SO(4)\rightarrow SO(3)\times SO(3)$. For a possible dashed arc, the
fundamental group of a regular neighborhood is generated by an order
$2$ element and an order $3$ element. Then notice that $T\cong A_4$,
$O\cong S_4$, $J\cong A_5$, by Lemma \ref{surjection}, any two such
elements in the groups $T\times_{C_3}T$, $O\times_{D_3}O$, $J\times
J$ cannot generate the whole group. Hence there is no allowable
dashed arc.
\end{proof}

\begin{lemma}\label{surjection}
Let $S$ be one of the permutation groups $A_4$, $S_4$, $A_5$. Let
$H$ be a subgroup of $S\times S$ such that the restrictions to $H$
of the two canonical projections of  $S\times S$ to $S$ are both
surjective. If $H$ is not isomorphic to $S$ then an order 2 element
and an order 3 element in $H$ cannot generate $H$.
\end{lemma}

\begin{proof}
Let $(x, x')$ and $(y, y')$ be order $2$ and order $3$ elements in
$H$ which generate $H$. Since the two projections restricted to $H$
are surjective, both $x$ and $x'$ have order $2$, and both $y$ and
$y'$ have order $3$; moreover the subgroups generated by $x, y$ and
also by $x', y'$ are both equal to $S$. One can check now by
explicit computations in each of the three groups that the  map
$x\mapsto x'$, $y\mapsto y'$ gives an isomorphism of $S$ to itself.
Hence $H$ is isomorphic to $S$.
\end{proof}

Now we proof our main results.

\vskip 0.1true cm

\begin{proof}
Theorem \ref{main table} follows from Theorem \ref{classify};
Theorem \ref{OE} and Theorem \ref{OES} follow from Theorem \ref{main
table}, with some elementary arithmetic.

Note that $4n(g-1)/(n-2)$ will be $12(g-1)$ when  $n=3$ and $8(g-1)$
when $n=4$; also,  $4n(g-1)/(n-2)$ will be $4(\sqrt{g}+1)^2$ when
$g=(n-1)^2$ and $4(g+1)$ when $g=n-1$.

Only the last two rows of the tables in Theorems \ref{OE} contain
infinitely many genera, corresponding to the orbifolds 15E and 19 in
Theorem \ref{classify}, or just corresponding to Examples 2.1 in
\cite{WWZZ}.

To derive Theorem \ref{OEK} we notice that by Example \ref{cage} we
have $OE^k_g\geq 4(g-1)$. And by Theorem \ref{main table}, we know
all cases with $|G|>4(g-1)$. Then we reach Theorem \ref{OEK}.
\end{proof}

\textbf{Comment:} We define two actions of a finite group $G$ to be
equivalent if the corresponding groups of homeomorphisms of $(S^3,
\Sigma_g)$ are conjugate (i.e., allowing isomorphisms of $G$). By
the proof of Theorem \ref{classify} and the tables above, we have:
There are only finitely many types of actions of $G$ on $(S^3,
\Sigma_g)$ such that $|G|>4(g-1)$ and the handlebody orbifold
bounded by $\Sigma_g/G$ is not of type II. In particular there are
only finitely many types of actions of $G$ on $(S^3, \Sigma_g)$
realizing $OE_g$ for $g\ne 21, 481$.

\section{Examples}

\begin{example}\label{cage}
Consider $S^3=\{(z_1,z_2)\in\mathbb{C}^2\mid |z_1|^2+|z_2|^2=1\}$ as
the unit sphere in $\mathbb{C}^2$. Let $a_j=(e^{\frac{2j\pi
i}{m}},0)$, $b_k=(0,e^{\frac{2k\pi i}{n}})$, here $j=0,1,\dots,m-1$,
$k=0,1,\dots,n-1$. Join each $a_j$ and $b_k$ by a shortest geodesic
arc in $S^3$, we get a two-parted graph $\Gamma\in S^3$ with $m+n$
vertices and $mn$ edges. Hence $\chi(\Gamma)=m+n-mn$. A regular
neighborhood of $\Gamma$ in $S^3$ is a handlebody of genus
$g=(m-1)(n-1)$, we denote its boundary by $\Sigma_g$.
See Figure 33 for $m=n=3$ and $g=4$.

\begin{center}
\scalebox{0.5}{\includegraphics{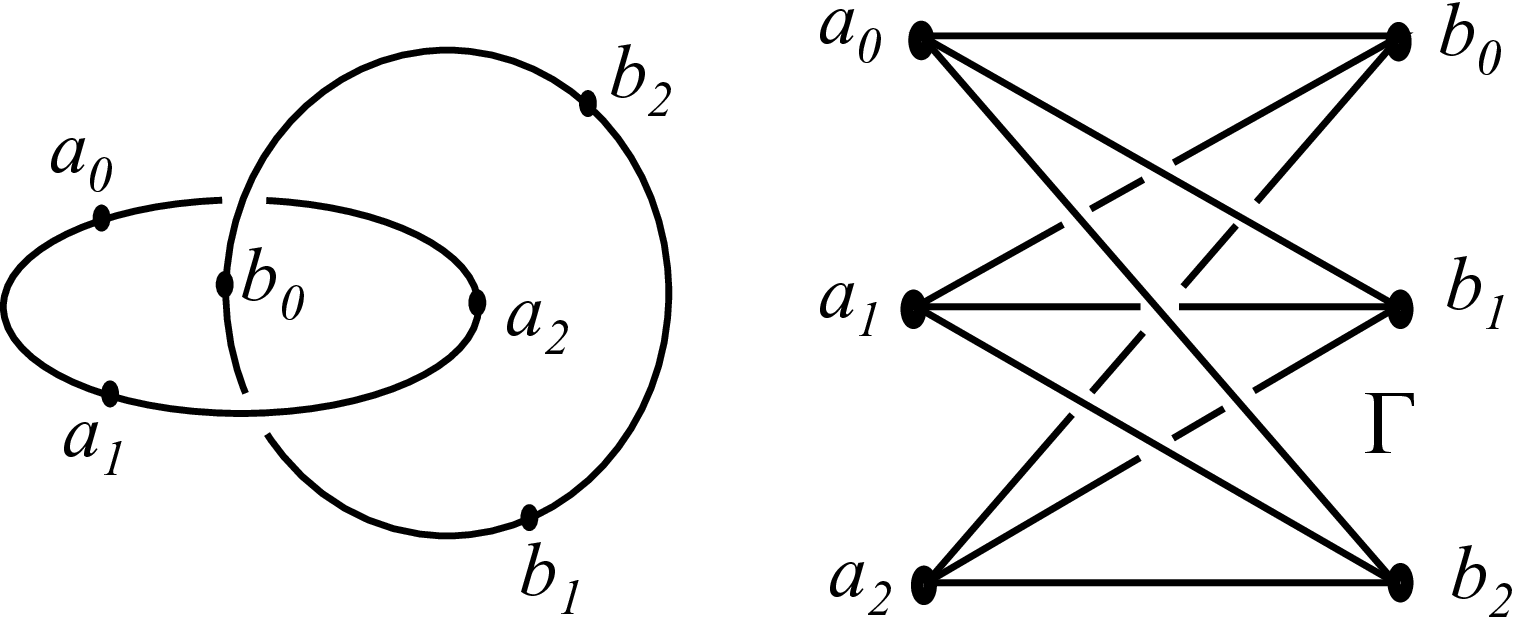}}

Figure 33
\end{center}

Then there is a group $G\cong(\Z_m\times\Z_n)\rtimes\Z_2$ acting on
$(S^3,\Sigma_g)$ given by:
$$x: (z_1,z_2)\mapsto(e^{\frac{2\pi i}{m}}z_1,z_2),$$
$$y: (z_1,z_2)\mapsto(z_1,e^{\frac{2\pi i}{n}}z_2),$$
$$t: (z_1,z_2)\mapsto(\bar{z}_1,\bar{z}_2).$$

When $m=2$, $g=n-1$, $G\cong D_{g+1}\times \Z_2$ acts on
$(S^3,\Sigma_g)$. This gives an extendable group action of order
$4(g+1)$, corresponding to orbifold 15E in the list of Theorem
\ref{classify}.

When $m=n$, $g=(n-1)^2$, there is another action of order $2$:
$$s: (z_1,z_2)\mapsto(z_2,z_1).$$
So there is a larger group $G=\langle
x,y,s,t\rangle\cong(\Z_n\times\Z_n)\rtimes(\Z_2\times\Z_2)$ acting
on $(S^3,\Sigma_g)$. This gives an extendable group action of order
$4(\sqrt{g}+1)^2$, corresponding to orbifold 19 in the list of
Theorem \ref{classify}.

The following picture shows how the orbifolds can be obtained.

\begin{center}
\scalebox{0.5}{\includegraphics{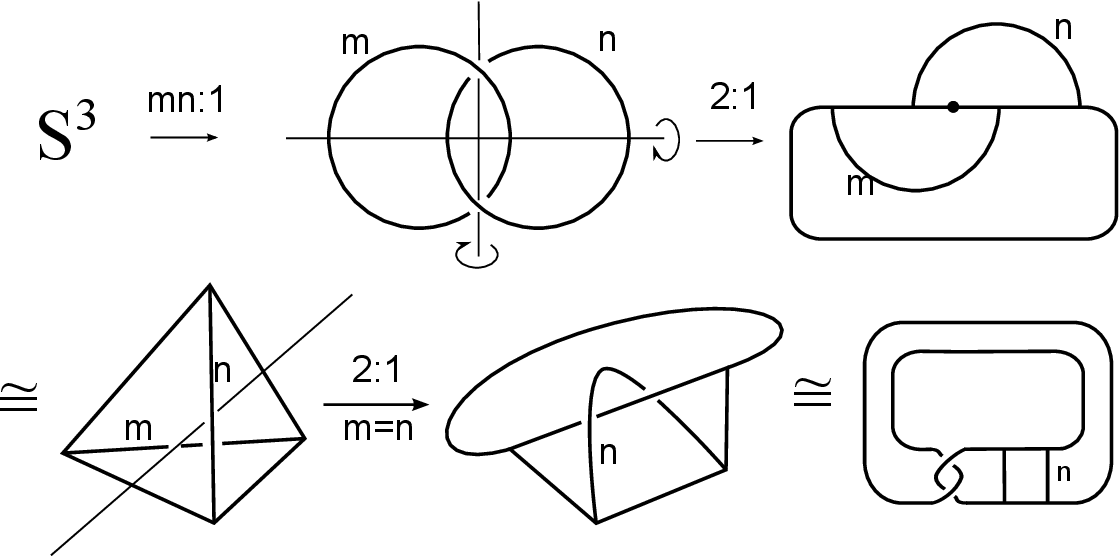}}

Figure 34
\end{center}
\end{example}

\begin{example}
If in the orbifold 15E We choose a dashed arc connecting a vertex
with incident edges labeled $(2, 2, 2)$ and an edge of index $2$. We
can knot this arc in an arbitrary way. The boundary of a regular
neighborhood of the arc is a knotted 2-suborbifold, and its preimage
in $S^3$ is connected. This gives us an order $4(g-1)$ extendable
$(D_{g-1}\oplus \Z_2)$-action on $\Sigma_g$, with respect to a
knotted embedding.

\begin{center}
\scalebox{0.6}{\includegraphics*[0pt,0pt][390pt,114pt]{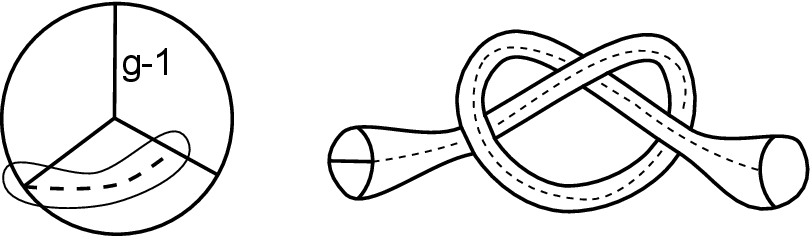}}

Figure 35
\end{center}
\end{example}

\begin{example}
Figure 37 shows a knotted handlebody of genus $g=11$ which is
invariant under a group action of order 120 of $S^3$, corresponding
to edge $c$ in 34 (Table VI). All points are colored by their
distance from the origin.

The group is isomorphic to $A_5\times \Z_2$, where we consider the
alternating group $A_5$ as the orientation preserving symmetry group
of the 4-dimensional regular Euclidean simplex, and $\Z_2$ is
generated by $-\text{id}$ on $E^4$. Let the 4-simplex be centered at
the origin of $E^4$ and inscribed in the unit sphere $S^3$. The
radial projection of its boundary to $S^3$ gives a tesselation of
$S^3$ by 5 tetrahedra invariant under the action of $A_5$. We
present one of these tetrahedron in Figure 36.

\begin{center}
\scalebox{0.6}{\includegraphics*[0pt,0pt][185pt,172pt]{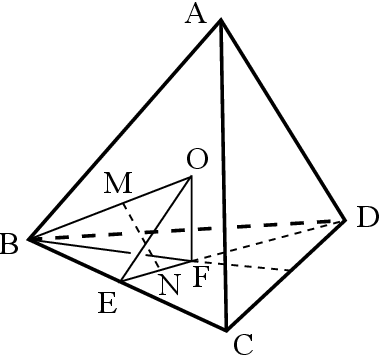}}

Figure 36
\end{center}

Imagine the figure has spherical geometry. $O$ is the center of the
tetrahedron, $F$ is the center of triangle $\triangle BCD$, $E$ is
the middle points of $BC$, $M$ is the middle points of $BO$, $N$ is
the middle points of $EF$. The orbit of the geodesic $MN$ under the
group action of $A_5$ joins to a graph in $S^3$; note that this
graph is invariant also under $-\text{id}$ on $S^3$. Projecting to
$E^3$, Figure 37 shows this graph and the boundary surface of the
regular neighborhood of the projected image(by \cite{Mathematica}).

\vskip 0.2true cm

\begin{center}
\scalebox{0.6}{\includegraphics*[0pt,0pt][520pt,565pt]{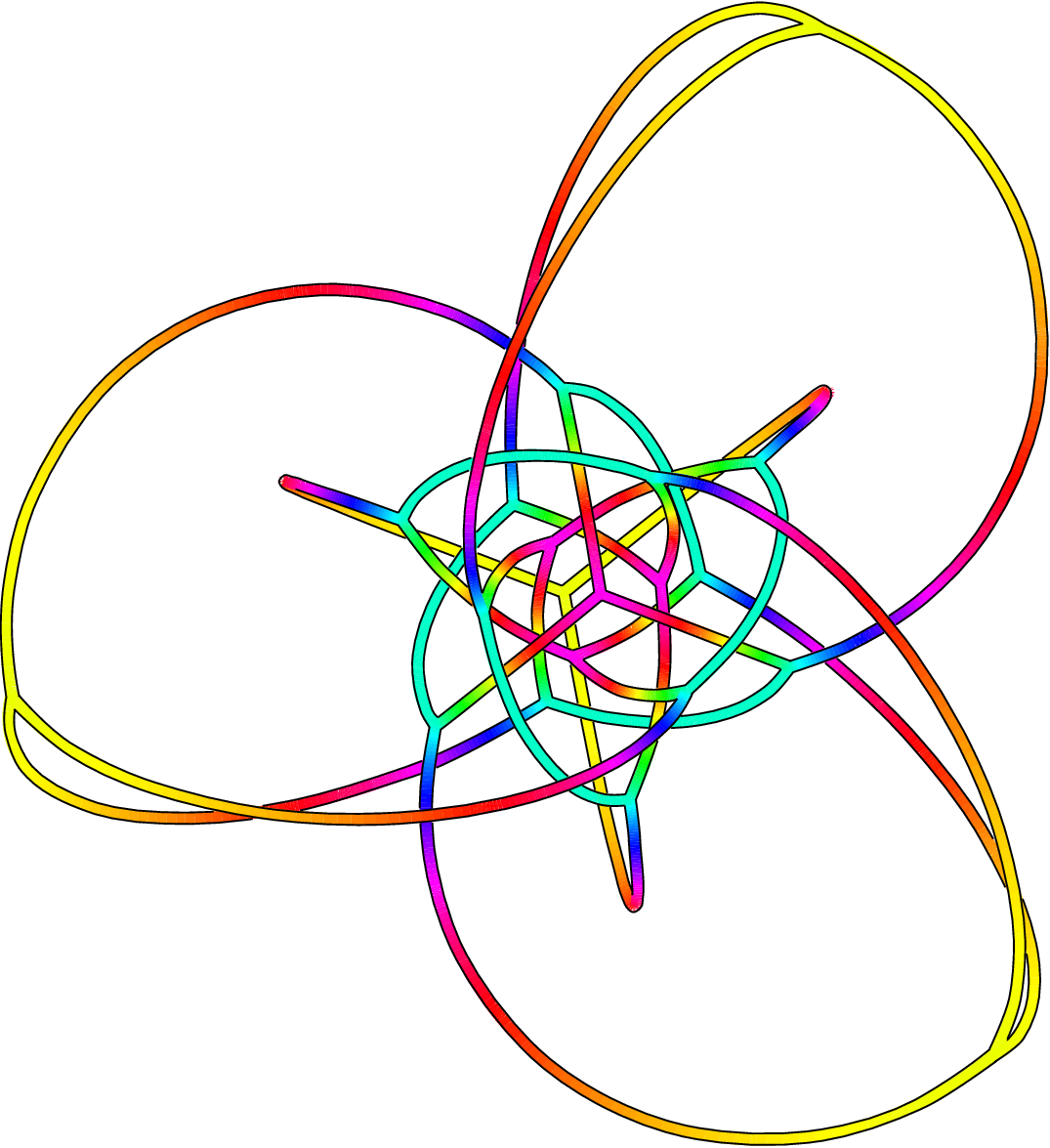}}

Figure 37
\end{center}
\end{example}

\section{Appendix}
In this section we give a very brief introduction on how to use the computer to examining the proofs given in Section 6.

One can download the GAP software in the ``download" section in

http://www.gap-system.org/

and in the ``Installation" section it shows how to run GAP on your computer.
On the same webpage, in ``Documentation$\rightarrow$Manuals" section, you can get all the details about how to use all the functions in GAP.

For example here we will explain how to input the data in Example 6.3 into GAP.
In the command line window of GAP, type
\beginexample
f:=FreeGroup("x", "y", "z");
x:=f.1;
y:=f.2;
z:=f.3;
G:=f/[x^2, y^3, z^2, (z*y)^2, (y*x*z)^2, (y*x*z*x)^3];
\endexample
to input the orbifold's fundamental group into GAP.
Then simply type
\beginexample
Size(G);
\endexample
then GAP will show that this group order is 120.
Then you can use
\beginexample
u:=G.1;
v:=G.2;
w:=G.3;
H:=GroupWithGenerators([v*u*w, u, w*v]);
\endexample
to input the subgroup standing for the arc orbifold into GAP.
Then use
\beginexample
Index(G, H);
\endexample
GAP will show that this is equal to 1.

All the usage of functions such as ``FreeGroup", ``GroupWithGenerators", ... can be found in ``Documentation$\rightarrow$Manuals$\rightarrow$Reference" in GAP's homepage.

\end{document}